\documentclass[10pt, twosides]{amsart}

\usepackage{defs, amsaddr}

\usepackage{thmtools}
\usepackage{graphicx}
\graphicspath{{./}}
\usepackage{caption}
\usepackage{subcaption}

\declaretheorem[style=remark]{example}
\renewcommand\thmcontinues[1]{Continued}
\DeclareMathOperator\erfc{erfc}

\title[Scenario Approach]{Scenario approach for minmax optimization with emphasis on the nonconvex case: positive results and caveats}

\author[M. Assif P. K., D. Chatterjee, and R. Banavar]{Mishal Assif P. K.}
\address{Department of Mechanical Engineering\\ IIT Bombay, Powai\\ Mumbai 400076, India\\ \url{https://mishalassif.github.io}}

\author{Debasish Chatterjee and Ravi Banavar}
\address{Systems \& Control Engineering\\ IIT Bombay, Powai\\ Mumbai 400076, India\\ \url{http://www.sc.iitb.ac.in/~chatterjee}\\\url{http://www.sc.iitb.ac.in/~banavar}}

\email{\{mishal\_assif, dchatter, banavar\}@iitb.ac.in}
\thanks{This work was supported in part by scholarships from the Ministry of Human Resource \& Development, Government of India. We thank Soumik Pal for helpful discussions.}


\begin{document}

	\begin{abstract}
		We treat the so-called \emph{scenario approach}, a popular probabilistic approximation method for robust minmax optimization problems via independent and indentically distributed (i.i.d) sampling from the uncertainty set, from various perspectives. The scenario approach is well-studied in the important case of convex robust optimization problems, and here we examine how the phenomenon of concentration of measures affects the i.i.d sampling aspect of the scenario approach in high dimensions and its relation with the optimal values. Moreover, we perform a detailed study of both the asymptotic behaviour (consistency) and finite time behaviour of the scenario approach in the more general setting of nonconvex minmax optimization problems. In the direction of the asymptotic behaviour of the scenario approach, we present an obstruction to consistency that arises when the decision set is noncompact. In the direction of finite sample guarantees, we establish a general methodology for extracting ``probably approximately correct'' type estimates for the finite sample behaviour of the scenario approach for a large class of nonconvex problems.
    \end{abstract}

	\maketitle

	\section{The problem, prior results, perspectives and prelude to our results}
		\label{sec:opt prob}
\subsection{Minmax optimization and scenario approximation}

The minmax optimization problem is typically phrased as follows: Let $\scOptDim$ be a positive integer and ($\uncertainSet, \genMetric$) be a metric space. Let $\scOptSet$ be a nonempty subset of $\R^{\scOptDim}$ and $( \uncertainSet, \borelsa{\uncertainSet}, \probMeasure )$ be a probability space where $\borelsa{\uncertainSet}$  is the Borel $\sigma$-algebra on $\uncertainSet$ induced by the metric $\genMetric$. Let $\scCost : \scOptSet \times \uncertainSet \lra\R$ be a lower semicontinous (l.s.c) function.\footnote{Recall that a function $F: \uncertainSet \lra\R$ is \emph{lower semicontinous} (l.s.c.) if every sublevel sets of $F$ is closed, i.e., $\set[\big]{z \in \uncertainSet \suchthat F(z) \leq t}$ is closed for all $t \in \R$.}
 We are interested in the following robust optimization problem:
\begin{equation}\label{eq:ROP}
\begin{aligned}
	\ropOptSolution \Let \inf_{\scOptElement \in \scOptSet}  \ \sup_{\uncertainParam \in \uncertainSet} \scCost(\scOptElement, \uncertainParam).
\end{aligned}
\end{equation}
Here on the one hand, $\scOptElement$ plays the role of the decision variable, and $\scOptSet$ is the set of variables from which a choice of one decision has to be made. On the other hand, $\uncertainParam$ plays the role of a parameter that affects the cost associated with each decision variable, and takes a fixed, albeit unknown, value in the set $\uncertainSet$. In problem \eqref{eq:ROP}, in effect, we pick a decision variable that incurs the least cost assuming that the worst possible value of $\uncertainParam$ corresponding to each value of the decision variable is realised. 

If $\uncertainSet$ is an infinite set, then the minmax optimization problem \eqref{eq:ROP} is an example of a semi-infinite optimization problem. Semi-infinite problems have been reported to be computationally intractable to solve in general \cite{ref:BenNem-98, ref:BenNem-99, ref:BenNemRoo-01}. Nevertheless, such optimization problems are of great importance in engineering, and, consequently, there is a natural need to find computationally tractable tight approximations to the problem \eqref{eq:ROP}. The central object of study in this work is the following approximation to $\eqref{eq:ROP}$:
\begin{equation}\label{eq:SOP}
\begin{aligned}
	\sopOptSolution[\sampleSize] \Let \inf_{\scOptElement \in \scOptSet}  \ \max_{i = 1, .., \sampleSize} \scCost(\scOptElement, \uncertainParam_i),
\end{aligned}
\end{equation}
where $(\uncertainParam_i)_{i=1}^{\sampleSize}$ is an independent and identially distributed (i.i.d) sequence of elements sampled from $\uncertainSet$. This approximation is also known as the \emph{scenario approximation} to the minmax optimization problem \eqref{eq:ROP} \cite{ref:CalCam-05}. We call each instance of the optimization problem \eqref{eq:SOP} corresponding to the sample \sample{i}{\sampleSize} a \emph{scenario optimization problem}. Observe that each scenario optimization problem is no longer semi-infinite since the inner maximum involves only finitely many variables. This makes the scenario optimization problem computationally more tractable, at least for moderate values of $\sampleSize$, than otherwise; and this is an attractive feature of \eqref{eq:SOP}.


\subsection{Desirable properties of scenario approximation}
Before proceeding further we record the two following definitions for future reference:
\begin{equation}
\begin{aligned}
	\marginalFunction(\scOptElement) & \Let \sup_{\uncertainParam \in \uncertainSet} \scCost(\scOptElement, \uncertainParam), \quad \text{ and} \\
	\sampledMargFn{\sampleSize}(\scOptElement) & \Let \max_{i = 1, .., \sampleSize} \scCost(\scOptElement, \uncertainParam_i).
\end{aligned}
\end{equation}
We note that $\sampledMargFn{\sampleSize}$ involves an abuse of notation since $\sampledMargFn{\sampleSize}$ depends on the sample $\sample{i}{\sampleSize}$ and not just on its size $\sampleSize$; however, we suppress the explicit mention of this dependence in the interest of brevity. It follows immediately from the definition that $\sampledMargFn{\sampleSize}(\scOptElement) \leq \marginalFunction(\scOptElement) \ \text{ for all } \scOptElement \in \scOptSet \text{ and } \sampleSize \in \N,$ which means that
\begin{equation}
\label{eq:obv bound}
	\sopOptSolution = \inf_{\scOptElement \in \scOptSet} \sampledMargFn{\sampleSize}(\scOptElement) \leq \inf_{\scOptElement \in \scOptSet} \marginalFunction(\scOptElement) = \ropOptSolution.
\end{equation}
In other words, the value of each scenario approximation \eqref{eq:SOP} is always an approximation of $\ropOptSolution$ from below. 

There naturally arises questions about the goodness of such approximations. One natural notion of goodness of the scenario approximation scheme is qualified in the form of:\par
\noindent\textbf{(G1) Consistency:} Recall that a numerical approximation procedure is said to be consistent if, intuitively as the level of approximation is made ``finer'', the approximate solution it computes converges to the actual solution of the problem being approximated. Consistency is a very rudimentary property that most sound numerical approximation procedures are expected to possess, and we say that the scenario approach is consistent if
\begin{equation*}
\prodProb{\probMeasure}{\infty}\left( \set[\Big]{ \sample{i}{+\infty} \suchthat \lim_{\sampleSize \to +\infty} \sopOptSolution[\sampleSize] = \ropOptSolution } \right) = 1.
\end{equation*}

To wit, the scenario approach is consistent if for almost every (countable) sequence of samples from $\uncertainSet$, the approximate solution computed by the scenario approach using only a finite inital segment of the sequence converges to the solution of the original problem with the length of the initial segment. We will study consistency of the scenario approach in greater detail in the later sections. In particular, we will establish an obstruction to consistency that arises when the set $\scOptSet$ is noncompact, a condition under which the scenario approach is \emph{guaranteed} to be inconsistent. 

A second desirable property of the scenario approximation is a good quality of:\par
\noindent \textbf{(G2) Finite sample behaviour:} Observe that the condition of consistency is purely asymptotic; it gives us no information about the nature of the approximate solution computed after 
drawing only a finite number of samples. But in the real world, information regarding the finite sample behaviour of the scenario approach is crucial and the behaviour of the scenario approach after drawing finitely many samples also warrants attention. In addition, it is desirable that this information is available to us \emph{a priori}, before the approximation procedure is executed so that the number of samples drawn can be determined based on the accuracy demanded by the application before executing the approximation scheme. We start our study of finite sample behaviour by quantifying levels of approximation and bad samples associated with scenario approximations. For $\scAccuracy > 0$ define the set of ``bad'' samples of size $\sampleSize$ as those that give at least $\scAccuracy$-bad estimates of $\ropOptSolution$, that is, those for which $\sopOptSolution$ is at least $\scAccuracy$ away from $\ropOptSolution$:
\begin{equation}
\label{eq:bad set def orig}
\begin{aligned}
 \badSet{\sampleSize}{\scAccuracy} \Let & \set[\Big]{ \sample{i}{\sampleSize} \in \uncertainSet^{\sampleSize} \suchthat \inf_{\scOptElement \in \scOptSet} \sampledMargFn{\sampleSize}(\scOptElement) \leq \ropOptSolution - \scAccuracy }. \\
\end{aligned}
\end{equation}
Defining
\begin{equation}
\label{eq:bad set aux def}
\begin{aligned}
\badSetU{\sampleSize}{\scAccuracy} \Let & \set[\Big]{ \sample{i}{\sampleSize} \in \uncertainSet^{\sampleSize} \suchthat \text{ there exist } \scOptElement \in \scOptSet \text{ such that } \sampledMargFn{\sampleSize}(\scOptElement) \leq \ropOptSolution - \scAccuracy} \\
= & \bigcup_{\scOptElement \in \scOptSet} \set[\Big]{ \sample{i}{\sampleSize} \in \uncertainSet^{\sampleSize} \suchthat \sampledMargFn{\sampleSize}(\scOptElement) \leq \ropOptSolution - \scAccuracy},
\end{aligned}
\end{equation}
we immediately get the sandwich relation
\begin{equation}\label{eq:bad set new}
\badSetU{\sampleSize}{\scAccuracy} \subset \badSet{\sampleSize}{\scAccuracy} \subset \badSetU{\sampleSize}{\frac{\scAccuracy}{2}}. 
\end{equation}
We find it easier to work with $\badSetU{\sampleSize}{\scAccuracy}$ than $\badSet{\sampleSize}{\scAccuracy}$, and since these two sets are sandwiched between each other, estimates for the probability of one will naturally lead to estimates for probability of the other. Note that there is nothing special about the factor $2\scAccuracy$ in \eqref{eq:bad set new}, the relation holds with any factor strictly greater than 1, we chose $2$ for convenience. Since
\begin{equation*}
\begin{aligned}
\set[\Big] { \sample{i}{\sampleSize} \in \uncertainSet^{\sampleSize} \suchthat \sampledMargFn{\sampleSize}(\scOptElement) \leq \ropOptSolution - \scAccuracy} &= \set[\Big]{ \sample{i}{\sampleSize} \in \uncertainSet^{\sampleSize} \suchthat \max_{i = 1, .. ,\sampleSize} \scCost(\scOptElement, \uncertainParam_i) \leq \ropOptSolution - \scAccuracy } \\
 &= \bigcap_{i = 1}^{\sampleSize} \set[\Big]{ \sample{i}{\sampleSize} \in \uncertainSet^{\sampleSize} \suchthat \scCost(\scOptElement, \uncertainParam_i) \leq \ropOptSolution - \scAccuracy},
\end{aligned}
\end{equation*}
the set defined in \eqref{eq:bad set aux def} is, in fact,
\begin{equation}
	\label{eq: bad set def}
\begin{aligned}
\badSet{\sampleSize}{\scAccuracy} = \bigcup_{\scOptElement \in \scOptSet} \bigcap_{i = 1}^{\sampleSize} \set[\Big]{ \sample{i}{\sampleSize} \in \uncertainSet^{\sampleSize} \suchthat \scCost(\scOptElement, \uncertainParam_i) \leq \ropOptSolution - \scAccuracy}. 
\end{aligned}
\end{equation}

Since we already have the obvious bound \eqref{eq:obv bound} that $\sopOptSolution[\sampleSize] \leq \ropOptSolution$, if $ \sopOptSolution[\sampleSize] > \ropOptSolution -\scAccuracy$, then we also get $| \sopOptSolution - \ropOptSolution | < \scAccuracy$. In other words, if a sample lies outside the bad set $\badSet{\sampleSize}{\scAccuracy}$, then the difference between the approximate infimum and the actual infimum is less than than $\scAccuracy$. Naturally, it is desirable that the bad set is as `small' as possible.  

As mentioned earlier, in the real world a priori quantitative information regarding the finite sample behaviour of the scenario approach is crucial. We are especially interested in results that provide an upper bound on the probability of occurence of the bad set $\badSet{\sampleSize}{\scAccuracy}$
\begin{equation}\label{eq: bad set prob}
\badSetProb{\sampleSize}{\scAccuracy} \Let \prodProb{\probMeasure}{\sampleSize}(\badSet{\sampleSize}{\scAccuracy}),
\end{equation}
before the approximation procedure begins. Such quantitative bounds provide a PAC (``probably approximately correct'') type guarantee that the scenario approximation \eqref{eq:SOP} computed using $\sampleSize$ i.i.d. samples from $\uncertainSet$ has an accuracy $\scAccuracy$ with probability at least as much as $1 - \badSetProb{\sampleSize}{\scAccuracy}$. Given an accuracy level $\scAccuracy$ and a confidence level $\beta$, from such a bound we may determine the number of samples that are required to ensure that the approximate minimum is at most $\scAccuracy$ away from the actual minimum with a probability at least $1 - \beta$. In the subsequent sections we will prove such PAC type guarantees for a large class of nonconvex minmax optimization programs. 

\subsection{A technical look at the sampling probability}

Since the scenario approximation procedure involves i.i.d sampling according to an arbitrary probability distribution, it is not reasonable to expect that the sampled maximum approximates the supremum. If the probability distribution has large ``holes'' in regions of $\uncertainSet$ where the supremum is achieved, these regions are never explored and consequently the sampled maximum does not approximate the supremum. We refer the reader to \cite[\S 1]{ref:Ram-18} for a more detailed discussion on this matter. A more meaningful notion of a supremum that we expect the sampled maximum to approximate is that of the essential supremum:
\begin{equation}
\esssup_{\uncertainParam \in \uncertainSet} \scCost(\scOptElement, \uncertainParam) := \inf \set[\Big]{ z \in \R \ \suchthat \ \probMeasure \big( \set{ \uncertainParam \in \uncertainSet \ \suchthat \ \scCost(\scOptElement, \uncertainParam) \leq z }\big) = 1 }.
\end{equation}
While the supremum of a set of numbers is its least upper bound, the essential supremum is the least ``almost'' upper bound. If the probability distribution has ``holes'' in certain regions, the essential supremum avoids considering the value of the function in these regions automatically, and therefore avoiding these regions while trying to approximate the essential supremum does not create any technical issues. However, the supremum posesses a lot of nice properties that the essentail supremum does not posess, and in order to avoid having to deal with the additional complications brought by the essential supremum, we would like to ensure that the supremum and the essential supremum are one and the same. Fortunately, it turns out that the assumption of lower semicontinuity of $\scCost$ that we made at the start is sufficient for this, and to verify this statement, we start with a standard definition from measure theory.
\begin{definition}[{{\cite[Definition 2.1, p.\ 28]{ref:Par-67}}}]
The support of the measure $\probMeasure$ is
\begin{equation*}
\mSupp{\uncertainSet} := \set[\Big]{ \uncertainParam \in \uncertainSet \ \suchthat \ \probMeasure \big(N_{\uncertainParam}\big) > 0 \text{ for every open neighbourhood } N_\uncertainParam \text{ of } \uncertainParam }.
\end{equation*}
It can be shown that $\mSupp{\uncertainSet}$ is the smallest (w.r.t set inclusion) closed subset of $\uncertainSet$ that has $\probMeasure$-measure 1.
\end{definition}

\begin{lemma}
\label{l:ess sup equivalence}
Consider the problem \eqref{eq:ROP} with its associated data. If $\scCost : \scOptSet \times \uncertainSet \lra\R$ is lower semicontinuous, then for each $\scOptElement \in \scOptSet$
\begin{equation*}
	\esssup_{\uncertainParam \in \uncertainSet} \scCost(\scOptElement, \uncertainParam) = \sup_{\uncertainParam \in \mSupp{\uncertainSet}} \scCost(\scOptElement, \uncertainParam).
\end{equation*}
\end{lemma}

A proof of this result is provided in Appendix \ref{app:proofs}. Lemma \ref{l:ess sup equivalence} says that lower semicontinuity of $\scCost$ ensures that the essential supremum is equal to the supremum on a certain subset of probability 1. For the sake of brevity of notation in the following discussion, we assume that $\mSupp{\uncertainSet} = \uncertainSet$; all the results that we derive below carry over to the situation when $\mSupp{\uncertainSet} \subsetneq \uncertainSet$. 
\begin{assumption}\label{ass:full support}
We stipulate that $\probMeasure$ is a fully supported probability measure, that is, 
\begin{equation*}
\mSupp{\uncertainSet} = \uncertainSet.
\end{equation*}
This is equivalent to stipulating that
\begin{equation*}
\probMeasure(U) > 0 \ \text{ for all open subsets } U \subset \uncertainSet.
\end{equation*}
\end{assumption}
Assumption \ref{ass:full support} ensures that it is not unreasonable to expect that the sampled maximum approximates the supremum, and this is the content of the next lemma: 
\begin{lemma}\label{l:ess sup}
Consider the problem \eqref{eq:ROP} with its associated data. If Assumption \ref{ass:full support} holds, then 
\begin{equation}
	\sup_{\uncertainParam \in \uncertainSet} \scCost(\scOptElement, \uncertainParam) = \esssup_{\uncertainParam \in \uncertainSet} \scCost(\scOptElement, \uncertainParam).
\end{equation}
\end{lemma}

All the assumptions made until this point will remain in force throughout the article, unless specifically mentioned otherwise; for the convenience of the reader, we recollect them here. 

	\begin{itemize}[label=\(\circ\), leftmargin=*]
		\item $\scOptSet$ is a subset of $\R^{\scOptDim}$ and $\uncertainSet$ is a metric space.
		\item $\scCost : \scOptSet \times \uncertainSet \lra\R$ is a lower semicontinous function. 
		\item $\probMeasure$ is a fully supported probability measure on $\uncertainSet$.
	\end{itemize}

\subsection{Prior work and contributions}

The scenario approach has been studied extensively in the literature in the particular case where $\scOptSet$ is a convex set and $\scCost(\cdot, \uncertainParam):\scOptSet \lra\R$ is a convex function for each $\uncertainParam \in \uncertainSet$. We will henceforth refer to the problem as a \emph{random convex program}. We review two recent representative results related to scenario approximations of random convex programs: one on consistency and the other on finite sample behaviour of the scenario approach, and compare the contributions of this article with the two results. We point out that both of these results rely crucially on the results established in \cite{ref:CalCam-05, ref:CalCam-06,ref:CamGar-08}. 

The first result from \cite{ref:Ram-18} establishes the consistency of the scenario approach for random convex programs, under an additional stipulation of an appropriate notion of coercivity on the class of functions $\scCost(\cdot, \uncertainParam):\scOptSet \lra \R$. Recall that if $\genMetSpace$ is a metric space a function $F: \genMetSpace \lra\R$ is \emph{weakly coercive} if for each $t \geq 0$ there exists a compact set $C_t \subset \genMetSpace$ such that the $t$-sublevel set of $F$ is contained in $C_t$, that is, $\set[\big] {x \in \genMetSpace \suchthat F(x) \leq t} \subset C_t$. In particular, if $\genMetSpace$ is compact, every function $F$ is weakly coercive.

\begin{theorem}[{{\cite[Theorem 14 on p.\ 154]{ref:Ram-18}}}]
\label{th: consistency convex}
Consider the problem \eqref{eq:ROP} with its associated data. Suppose there exists an $\bar{N} \in \N$ such that 
\begin{equation}
\label{eq: eq coercivity ramponi}
\prodProb{\probMeasure}{\bar{N}} \left( \set[\Big]{ \sample{i}{\bar{N}} \suchthat \sampledMargFn{\bar{N}} \text{ is weakly coercive} } \right) > 0.
\end{equation}
In addition, if $\scCost(\cdot, \uncertainParam):\scOptSet \lra\R$ is convex for all $\uncertainParam \in \uncertainSet$, the scenario approach is consistent in the sense that
\begin{equation*}
\prodProb{\probMeasure}{\infty}\left( \set[\Big]{ \sample{i}{+\infty} \suchthat \lim_{\sampleSize \to +\infty} \sopOptSolution[\sampleSize] = \ropOptSolution } \right) = 0.
\end{equation*}
\end{theorem}

Under the additional Assumption of \eqref{eq: eq coercivity ramponi}, Theorem \ref{th: consistency convex} establishes the consistency of the scenario approach for random convex programs. When the set $\scOptSet$ itself is compact \eqref{eq: eq coercivity ramponi} holds trivially since any function on a compact set is weakly coercive. If $\scOptSet$ is non-compact \eqref{eq: eq coercivity ramponi} may fail to hold, and consequently the consistency of the scenario approach may also be jeopardized. We study this situation in detail, and the \emph{\textbf{first main contribution}} of the article will be identifying an obstruction to consistency of the scenario approach when $\scOptSet$ is non-compact, that is, a condition that guarantees that the scenario approach will not be consistent.

We now review the main result from \cite{ref:EsfSutLyg-15} that establishes finite sample guarantees for the scenario approach applied to random convex programs.

\begin{theorem}[{{\cite[Theorem 14 on p.\ 5]{ref:EsfSutLyg-15}}}]
\label{th:convex sop ftg}
The \emph{tail probability for worst case violation} is the function $\tailProbWors : \scOptSet \times \Rpos \lra [0, 1]$ defined by
\begin{equation}
\label{eq:tail prob worse}
\tailProbWors(\scOptElement, \scAccuracy) = \probMeasure \left( \set[\Big] { \uncertainParam \in \uncertainSet \suchthat \scCost(\scOptElement, \uncertainParam) > \marginalFunction(\scOptElement) - \scAccuracy  } \right).
\end{equation}
Moreover, let
\begin{equation}
\label{eq:inf tail prob worse}
 \infTailProbWors(\scAccuracy) = \inf_{\scOptElement \in \scOptSet} \tailProbWors(\scOptElement, \scAccuracy).
\end{equation}
The function $\unifLowBound : [0, 1] \lra \Rpos$ is called a \emph{uniform level set bound (ULB) of $\infTailProbWors$} if for every $\scAccuracy \in [0, 1]$,
\begin{equation}
\label{eq: ULB}
\unifLowBound(\scAccuracy) \geq \sup \set[\Big] { \delta \geq 0 \suchthat \infTailProbWors(\scAccuracy) \leq \delta}.
\end{equation} 
Define
\begin{equation}
\label{eq:def N}
N(\scAccuracy, \beta) \Let \min \set[\bigg] { N \in \N \suchthat \sum_{i = 0}^{\scOptDim - 1} {N \choose i} \scAccuracy^i (1-\scAccuracy)^{N-i} \leq \beta }.
\end{equation}
Given a ULB $h$ and numbers $\scAccuracy, \beta \in\: ]0, 1]$, for all $N \geq N(\scAccuracy, \beta)$ we have
\begin{equation}
\prodProb{\probMeasure}{N} \left( \set[\Big]{ \sample{i}{N} \in \uncertainSet^{N} \suchthat \ropOptSolution - \sopOptSolution \leq h(\scAccuracy) } \right) \geq 1-\beta.
\end{equation}

\end{theorem}

Theorem \ref{th:convex sop ftg} provides a guarantee that if $N(\scAccuracy, \beta)$ number of points are sampled in an i.i.d fashion from $\uncertainSet$ and the corresponding scenario approximation problem is solved to obtain an approximate minimum, then one can say with confidence $1-\beta$ that the approximate minimum $\sopOptSolution[N(\scAccuracy, \beta)]$ is at most $h(\scAccuracy)$ away from the true minimum $\ropOptSolution$. Note that the guarantee is a priori: one does not need any information related to the actual samples drawn $\sample{i}{\sampleSize}$ in order to compute $N(\scAccuracy, \beta)$, and consequently, one can use Theorem \ref{th:convex sop ftg} to determine the number of samples required to be drawn in order to obtain a solution of the given accuracy fixed at the beginning of the optimization procedure. One of the crucial ingredients in the proof of Theorem \ref{th:convex sop ftg} is a result from \cite{ref:CamGar-08} which is valid only for random convex programs. In the light of recent extensions of \cite{ref:CamGar-08} to the nonconvex case in \cite{ref:CamGarRam-18}, one can extend some results of \cite{ref:EsfSutLyg-15}, including Theorem \ref{th:convex sop ftg}, to nonconvex robust optimization problems. However, the results of \cite{ref:CamGarRam-18} in the nonconvex regime are of an \emph{a posteriori} nature, meaning that the guarantees given depend on the sample $\sample{i}{\sampleSize}$ drawn, and the extension of Theorem \ref{th:convex sop ftg} to the nonconvex case via that route inherits this same a posteriori property. This means that one cannot determine the number of samples that give an approximate solution of desired accuracy before the approximation procedure begins. However, once a sample $\sample{i}{\sampleSize}$ is drawn and the correspoding scenario approximation is found, then one can find the accuracy of the computed approximate solution. In other words, one can only assess the quality of a scenario approximate solution \emph{after} the solution is computed. In contrast, the \emph{\textbf{second main contribution}} of the present article is a methodology to establish a priori PAC type finite sample guarantees similar to Theorem \ref{th:convex sop ftg} that is applicable to scenario approximations of a large class of nonconvex minmax optimization problems.

\subsection{Numerical experiments in high dimensions}


We devote this subsection to examine in detail, with the aid of numerical experiments, a simple minmax optimization problem and the quality of scenario approximations of it. Recall that for a given vector $v \in \R^{n}$ the quantity $\infnorm{v}$ denotes its infinity norm defined by $\infnorm{v} = \max\limits_{i = 1,\ldots, n} |v_i|.$ Consider the optimization problem:
\begin{equation}\label{eq:num exp}
\ropOptSolution = \inf_{\scOptElement \in \left[0, 1\right]} \ \sup_{\uncertainParam \in \left[-1, 1\right]^{\uncertainDim}} \left( \scOptElement\infnorm{\uncertainParam} - \infnorm{\uncertainParam}^2 \right).
\end{equation}

\begin{figure}[t]
\centering
\includegraphics[height=0.8\textwidth, width=0.9\textwidth]{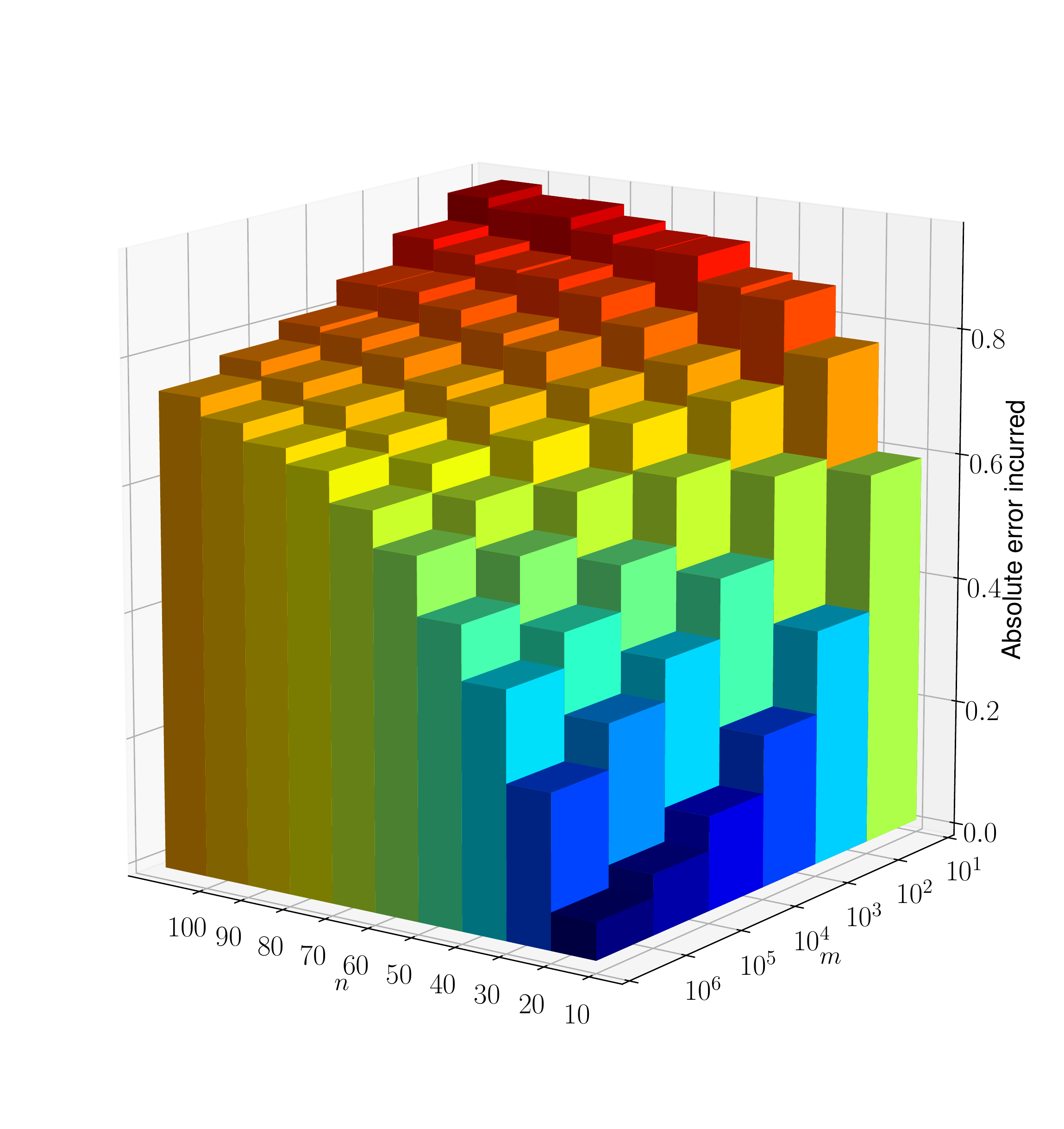}
\caption{Variation of the error in scenario approximations of the problem \eqref{eq:num exp} with respect to the dimension ($\uncertainDim$) of the uncertainty set and number ($\sampleSize$) of i.i.d samples drawn from the uncertainty set. The i.i.d samples were drawn according to the uniform distribution on $\uncertainSet = \left[-1, 1 \right]^{\uncertainDim}$. The numbers reported here correspond to the average error of the scenario approximations over 25 sets of samples of length $\sampleSize$ from $\uncertainSet = \left[-1, 1 \right]^{\uncertainDim}$.}
\label{fig:num exp unif 3d}
\end{figure}

\begin{figure}[t]
\centering
\includegraphics[width=10cm]{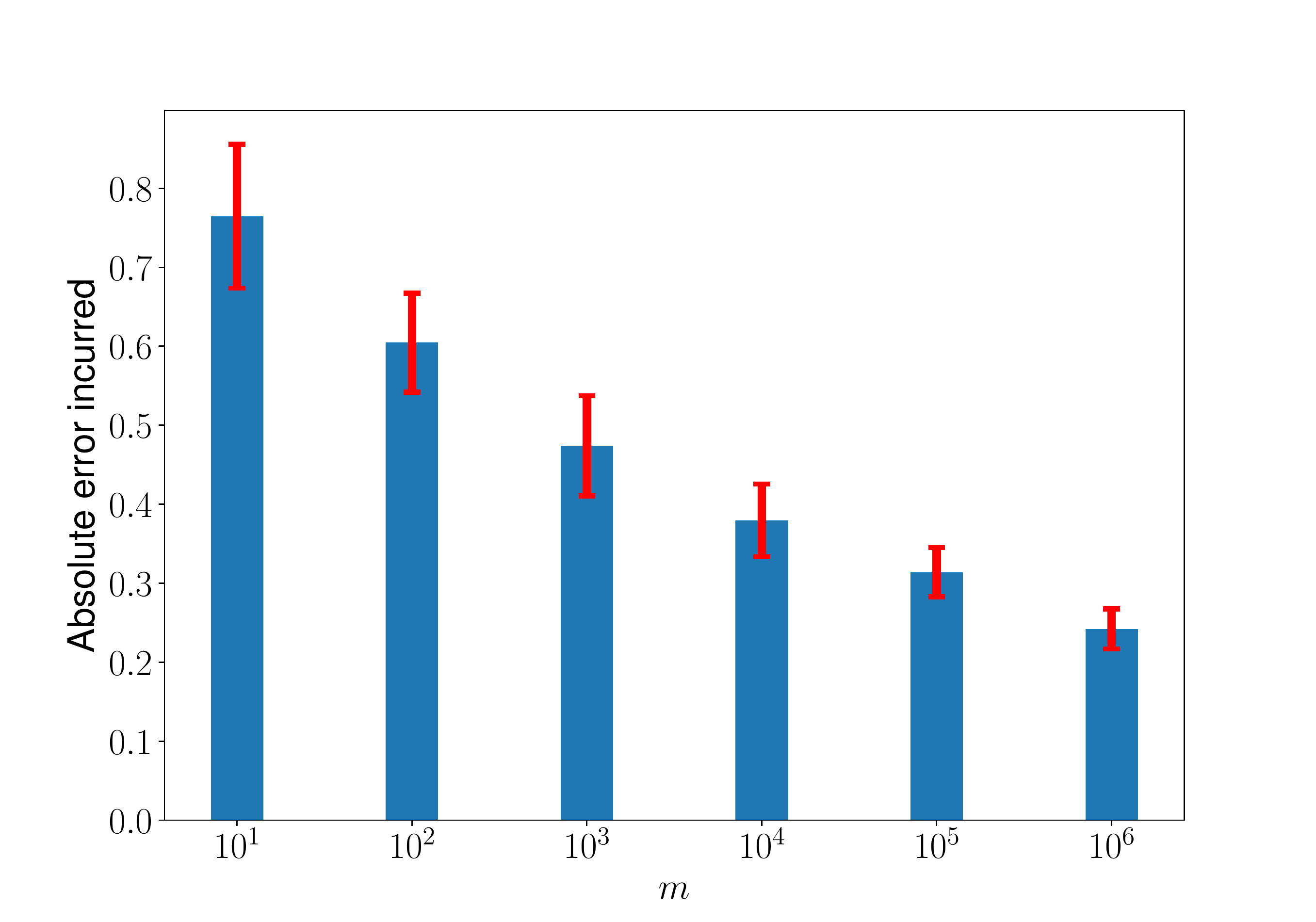}
\caption{Variation of the error in scenario approximations of the problem \eqref{eq:num exp} with respect to the number ($\sampleSize$) of i.i.d samples drawn from the uncertainty set when the dimension ($\uncertainDim$) of the uncertainty set is fixed at 20. The i.i.d samples were drawn according to the uniform distribution on $\uncertainSet = \left[-1,1 \right]^{\uncertainDim}$. The numbers reported here correspond to the average error of the scenario approximations over 25 sets of samples of length $\sampleSize$ from $\uncertainSet = \left[-1, 1\right]^{\uncertainDim}$.}
\label{fig:num exp unif 2d}
\end{figure}

In the language of the problem \eqref{eq:ROP}, here we have chosen $\scOptSet = \left[0, 1\right]$, $\uncertainSet = \left[-1, 1\right]^{\uncertainDim}$, and the continuous function $\scCost(\scOptElement, \uncertainParam) = \scOptElement\infnorm{\uncertainParam} - \infnorm{\uncertainParam}^2.$ Observe that for each $\uncertainParam \in \uncertainSet$, $\scCost(\cdot, \uncertainParam)$ is a convex function on the convex set $\scOptSet$. In other words, \eqref{eq:num exp} is a random convex program. Moreover the set $\scOptSet$ is compact and therefore \eqref{eq:num exp} satisfies all the conditions of Theorem \ref{th: consistency convex}, and the latter guarantees that the scenario approximations will almost surely converge to $\ropOptSolution$. 

To study the finite sample behaviour of scenario approximations of \eqref{eq:num exp}, we first compute the optimal value $\ropOptSolution$. This can be done by observing that
\begin{equation}
\sup_{\uncertainParam \in \left[-1, 1\right]^{\uncertainDim}} \left( \scOptElement\infnorm{\uncertainParam} - \infnorm{\uncertainParam}^2 \right) = \sup_{\infnorm{\uncertainParam} \in \left[-1, 1\right]} \left( \scOptElement\infnorm{\uncertainParam} - \infnorm{\uncertainParam}^2 \right) = \frac{\scOptElement^2}{4}.
\end{equation}
Consequently,
\begin{equation}
\ropOptSolution = \inf_{\scOptElement \in \left[0, 1\right]} \ \sup_{\uncertainParam \in \left[-1, 1\right]^{\uncertainDim}} \left( \scOptElement\infnorm{\uncertainParam} - \infnorm{\uncertainParam}^2 \right) = \inf_{\scOptElement \in \left[0, 1\right]} \ \frac{\scOptElement^2}{4} = 0.
\end{equation}
Since we have the optimal value $\ropOptSolution$ in \eqref{eq:num exp} and the scenario approximate solution $\sopOptSolution[\sampleSize]$ can be computed numerically on a computer, we can compute the error associated with the scenario approximations of \eqref{eq:num exp}. In Figure \ref{fig:num exp unif 3d} we present the results of our numerical experiments that give the error in the scenario approximation \eqref{eq:SOP} of \eqref{eq:num exp} and its variation with the dimension $\uncertainDim$ of the uncertainty set and the number $\sampleSize$ of samples. We sampled independently from the uniform distribution on $\uncertainSet = \left[-1, 1\right]^{\uncertainDim}$ to obtain these scenario approximations. The error shown in the figure for each value of $\sampleSize$ and $\uncertainDim$ was computed by taking the average error of the scenario approximations over 25 sets of samples of length $\sampleSize$ from $\uncertainSet = \left[-1, 1\right]^{\uncertainDim}$.

Figure \ref{fig:num exp unif 3d} follows the expected trend that the error decreases as the number of samples increases and the dimension of the uncertainty set decreases. A closer look at the value of the error shows that even for a moderate dimension of $\uncertainDim = 20$ (see Figure \ref{fig:num exp unif 2d}) of the uncertainty set, even after sampling as much as a million scenarios, one still gets an error as large as \(0.25\). To put this in perspective, observe that the value of the cost $\scCost(\scOptElement, \uncertainParam)$ varies between \(-1\) and \(+1\) as $\scOptElement$ and $\uncertainParam$ vary over $\scOptSet$ and $\uncertainSet$, respectively, which means that this is an error of about \(12.5\%\). The results are much worse for higher dimensions; for instance, when the dimension of the uncertainty set is \(50\) and a million scenarios are drawn, the error in the scenario approximation is around \(0.5\) in absolute units, which puts the relative error at around \(25\%\).

\begin{figure}[t]
\centering
\includegraphics[height=0.8\textwidth, width=0.9\textwidth]{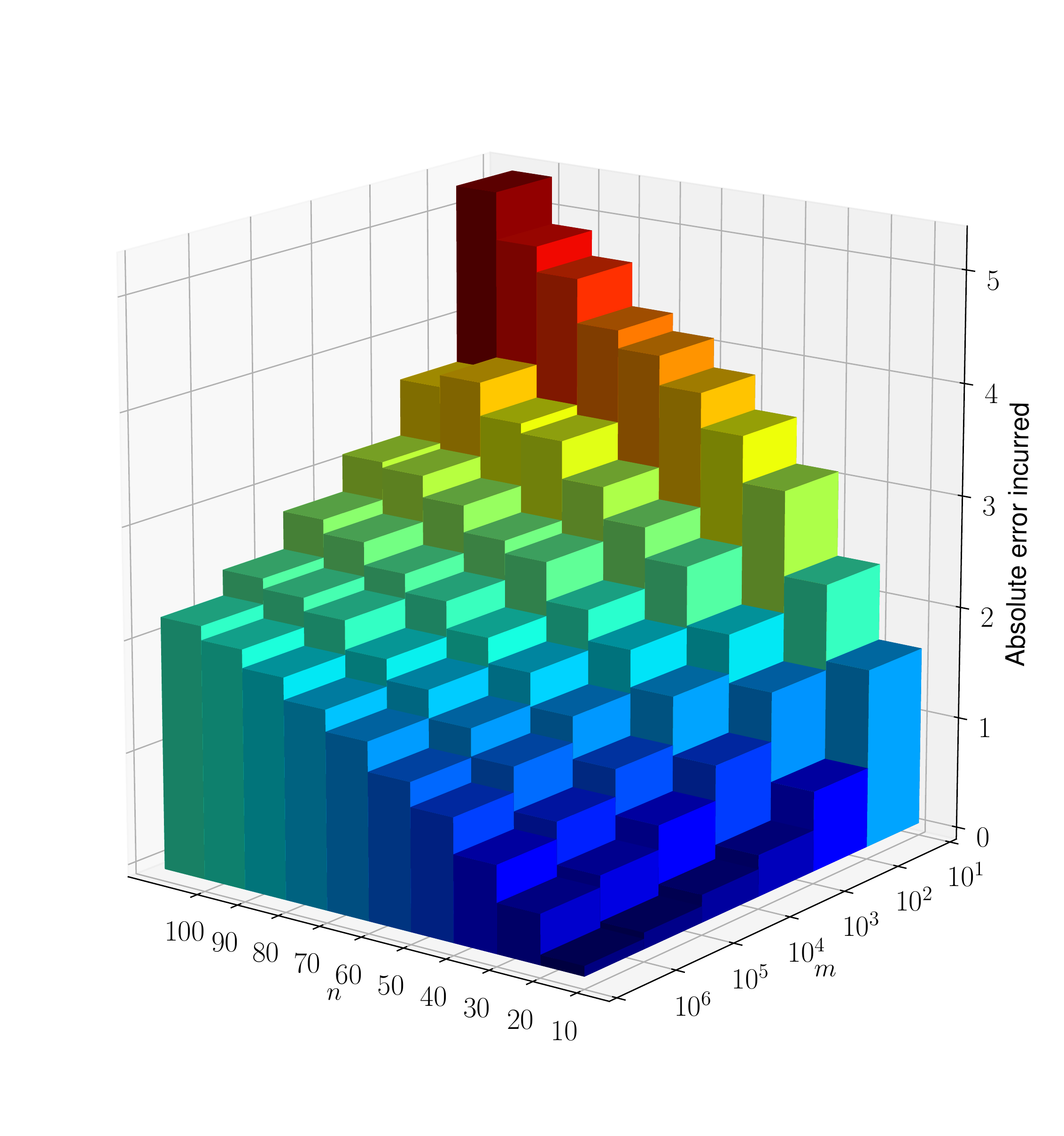}
\caption{Variation of the error in scenario approximations of the problem \eqref{eq:num exp 2} with respect to the dimension ($\uncertainDim$) of the uncertainty set and number ($\sampleSize$) of i.i.d samples drawn from the uncertainty set. The i.i.d samples were drawn according to the Gaussian distribution with mean 0 and variance $I_\uncertainDim$ on $\uncertainSet = \R^{\uncertainDim}$. The numbers reported here correspond to the average error of the scenario approximations over 25 sets of samples of length $\sampleSize$ from $\uncertainSet = \R^{\uncertainDim}$.}
\label{fig:num exp gauss 3d}
\end{figure}

\begin{figure}[t]
\centering
\includegraphics[width=10cm]{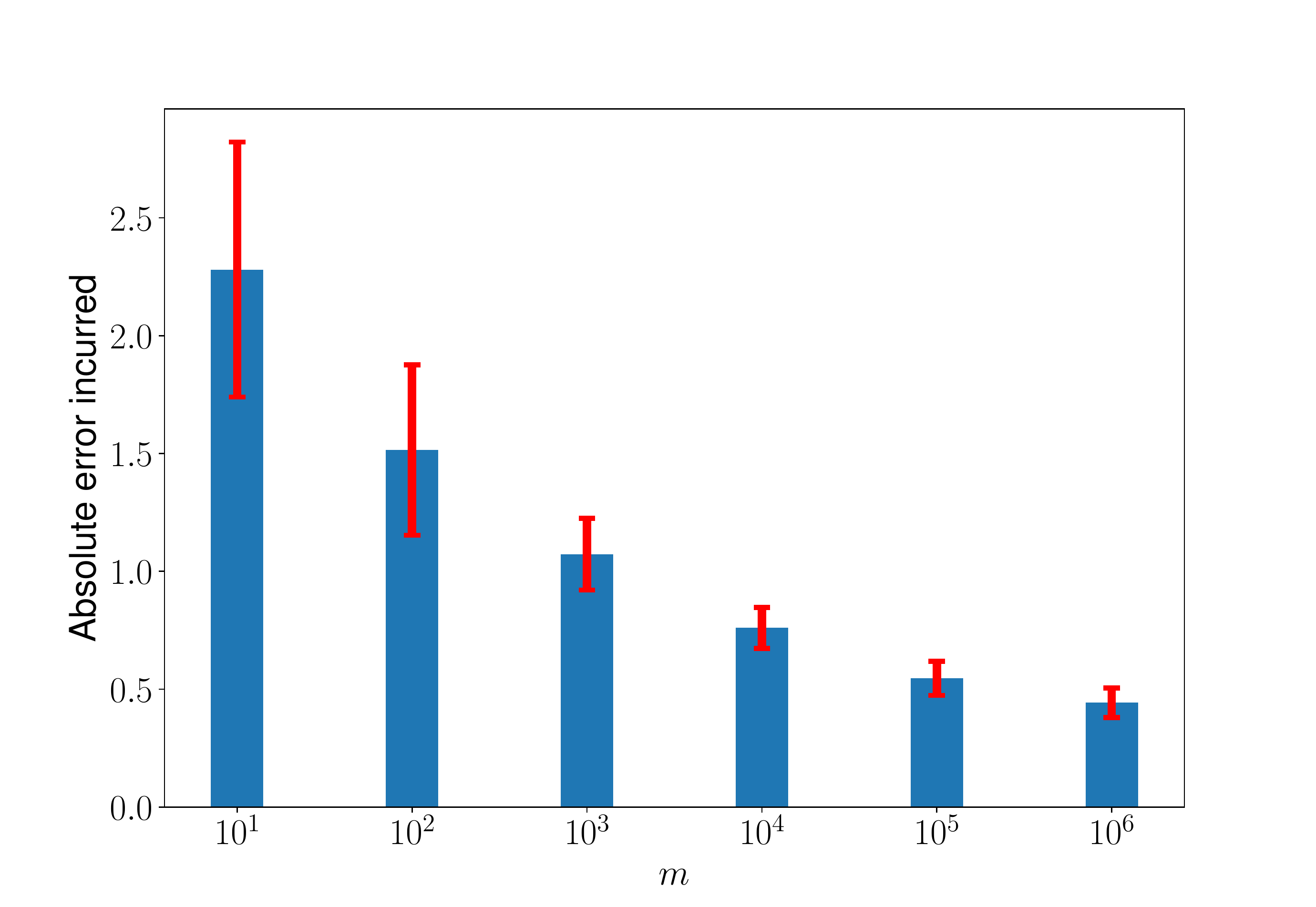}
\caption{Variation of the error in scenario approximations of the problem \eqref{eq:num exp 2} with respect to the number ($\sampleSize$) of i.i.d samples drawn from the uncertainty set when the dimension ($\uncertainDim$) of the uncertainty set is fixed at 20. The i.i.d samples were drawn according to the Gaussian distribution with mean 0 and variance $I_\uncertainDim$ on $\uncertainSet = \R^{\uncertainDim}$. The numbers reported here correspond to the average error of the scenario approximations over 25 sets of samples of length $\sampleSize$ from $\uncertainSet = \R^{\uncertainDim}$.}
\label{fig:num exp gauss 2d}
\end{figure}

We get even worse results if we consider the slightly modified problem
\begin{equation}\label{eq:num exp 2}
\ropOptSolution = \inf_{\scOptElement \in \left[0, 1\right]} \ \sup_{\uncertainParam \in \R^{\uncertainDim}} \left( \scOptElement\infnorm{\uncertainParam} - \infnorm{\uncertainParam}^2 \right),
\end{equation}
where the uncertainty set is noncompact. The optimal value is $\ropOptSolution = 0$ in this case as well; indeed,
\begin{equation}
\sup_{\uncertainParam \in \R^{\uncertainDim}} \left( \scOptElement\infnorm{\uncertainParam} - \infnorm{\uncertainParam}^2 \right) = \sup_{\infnorm{\uncertainParam} \in \R} \left( \scOptElement\infnorm{\uncertainParam} - \infnorm{\uncertainParam}^2 \right) = \frac{\scOptElement^2}{4},
\end{equation}
and consequently,
\begin{equation}
\ropOptSolution = \inf_{\scOptElement \in \left[0, 1\right]} \ \sup_{\uncertainParam \in \R^{\uncertainDim}} \left( \scOptElement\infnorm{\uncertainParam} - \infnorm{\uncertainParam}^2 \right) = \inf_{\scOptElement \in \left[0, 1\right]} \ \frac{\scOptElement^2}{4} = 0.
\end{equation}

In Figure \ref{fig:num exp gauss 3d} we present the results of our numerical experiments that give the error in the scenario approximation \eqref{eq:SOP} of \eqref{eq:num exp 2} and its variation with the dimension $\uncertainDim$ of the uncertainty set and the number $\sampleSize$ of samples. We sampled independently from the Gaussian distribution with mean 0 and variance $I_\uncertainDim$ on $\uncertainSet = \R^{\uncertainDim}$ to obtain these scenario approximations. The error shown in the figure for each value of $\sampleSize$ and $\uncertainDim$ was computed by taking the average error of the scenario approximations over \(25\) sets of samples of length $\sampleSize$ from $\uncertainSet = \R^{\uncertainDim}$. We see that even for a moderate dimension of $\uncertainDim = 20$ (see Figure \ref{fig:num exp gauss 2d}) of the uncertainty set, even after sampling as much as a million scenarios, one still gets an error as large as \(0.45\) in absolute units. As expected, the results are much worse for higher dimensions: for instance, when the dimension of the uncertainty set is \(50\) and a million scenarios are drawn, the error in the scenario approximation is around \(1.3\) in absolute units.

Of course, the measuring stick in the scenario approximations of the example immediately above the current paragraph is the probability measure corresponding to the Gaussian employed for sampling, and the specific concentration properties of this measure naturally affects the outcome of the experiment as a consequence. Whether these estimates are satisfactory or not is difficult to assess unilaterally and uniformly across the spectrum of robust minmax optimization problems, and such conclusions are best left to the judgment of the practitioners concerned.

The main culprit in the examples above is the fact that scenario approximations rely on i.i.d samples, and i.i.d samples of high dimensional random vectors tend to concentrate with high probability around certain regions of the space leaving the rest of the space unexplored; this feature leads to a preference for certain (typically thin) regions of the sample space of the algorithm, and unless the optimizers are in these thin sets, the quality of approximation may be low. The preceding observations clearly point to the fact that there is still scope to develop general, computationally feasible, and tight approximation schemes for robust optimization problems, \emph{especially in high dimensions insofar as the optimal value is concerned}; one such approximation method involving better sampling will be reported subsequently elsewhere.

	\section{An obstruction to consistency}
		\label{sec:negative}
Consistency of the scenario approach is not guaranteed, in general, for all problems of the form \eqref{eq:ROP}. Even in the particular case of random convex programs, observe that the statement of Theorem \ref{th: consistency convex} has the additional requirement of coercivity, which is not always satisfied if the set $\scOptSet$ is not compact. We begin with a simple example that illustrates this effect.

\begin{exmpl}\label{exa:cont}
 \label{eg:inconsistent}
 Let $\scOptSet = \R$ and $\uncertainSet = \R$. Assume that $\uncertainSet$ is endowed with the standard Gaussian probability measure with mean 0 and variance 1. Consider, in the language of \eqref{eq:ROP} the cost function
 \begin{equation}
 \label{eq:inconsistent}
 \scCost(\scOptElement, \uncertainParam) \Let
  \begin{cases} 
      0 & \text{if } \scOptElement \geq \uncertainParam, \\
      \uncertainParam - \scOptElement & \text{if } \uncertainParam-1\leq \scOptElement \leq \uncertainParam, \\
      -1 & \text{if } \scOptElement \leq \uncertainParam-1. 
   \end{cases}
 \end{equation}
 	All the requirements of the problem \eqref{eq:ROP} are satisfied by \eqref{eq:inconsistent} in addition to Assumption \ref{ass:full support}. Yet we show that the scenario approach is inconsistent in this situation. For a given sample $\sample{i}{\sampleSize}$ we define $\hat{\uncertainParam}_{\sampleSize} \Let \min_{i = 1,...,\sampleSize} \uncertainParam_i$. One checks that
 \begin{equation}
 \begin{aligned}
	\marginalFunction(\scOptElement) & \Let \sup_{\uncertainParam \in \uncertainSet} \scCost(\scOptElement, \uncertainParam) = 0, \quad \text{ and} \\
	\sampledMargFn{\sampleSize}(\scOptElement) & \Let \max_{i = 1, .., \sampleSize} \scCost(\scOptElement, \uncertainParam_i) = \scCost(\scOptElement, \hat{\uncertainParam}_{\sampleSize}), 
 \end{aligned}
 \end{equation}
 which imply that
 \begin{align}
 	\ropOptSolution &= \inf_{\scOptElement \in \scOptSet} \marginalFunction(\scOptElement) = 0, \quad \text {and} \\
 	\sopOptSolution &= \inf_{\scOptElement \in \scOptSet} \sampledMargFn{\sampleSize}(\scOptElement) = -1 \text{ for all $\sampleSize \in \N$}.
\end{align}
This means that $\lim_{\sampleSize \to +\infty} \sopOptSolution = -1$ for any sequence of samples $\sample{i}{+\infty}$, which shows that consistency fails to hold.
\end{exmpl}  

 It is clear from Example \ref{eg:inconsistent} that the set $\scOptSet$ being noncompact can readily lead to inconsistency of the scenario approach. In this section we study this issue further and characterize one possible obstruction to the consistency of the scenario approach when the set of optimization variables $\scOptSet$ is noncompact. We begin with the following definition that will be need in both the current and the next section:
\begin{definition}
The \emph{ tail probability } is the function $\tailProb: \scOptSet \times \Rpos \lra [0, 1]$ defined by
\begin{equation*}
	\tailProb(\scOptElement, \scAccuracy) \Let \probMeasure \bigl( \set{ \uncertainParam \in \uncertainSet \suchthat \scCost(x, \uncertainParam) > \ropOptSolution - \epsilon}  \bigr).
\end{equation*}
For each $\scAccuracy > 0$, we define the infimum of $\tailProb(\scOptElement, \scAccuracy)$ over all $\scOptElement \in \scOptSet$ by
\begin{equation*}
 \infTailProb(\scAccuracy) \Let \inf_{\scOptElement \in \scOptSet} \tailProb(\scOptElement, \scAccuracy).
\end{equation*}
\end{definition}

The following theorem is the key result of this section.
\begin{theorem}\label{th:weak nec}
	Consider the problem \eqref{eq:ROP} with its associated data, and suppose that Assumption \ref{ass:full support} holds. If there exists some $\scAccuracy > 0$ satisfying $\infTailProb(\scAccuracy) = 0$, then 
	\begin{equation*}
			\prodProb{\probMeasure}{\infty}\left( \set[\Big]{ \sample{i}{+\infty} \suchthat \lim_{\sampleSize \to +\infty} \sopOptSolution[\sampleSize] = \ropOptSolution } \right) = 0.
	\end{equation*}
\end{theorem}
\begin{proof}
Since $\sopOptSolution$ is monotone non-decreasing in $\sampleSize$, recalling the definition of $\badSet{\sampleSize}{\scAccuracy}$ from \eqref{eq: bad set def} we see that
	\begin{align*}
		\set[\Big]{ \sample{i}{\infty} \in \uncertainSet \suchthat \lim_{\sampleSize \to +\infty} \sopOptSolution[\sampleSize] \leq \ropOptSolution - \scAccuracy} &= \bigcap_{\sampleSize \in \Nz} \badSet{\sampleSize}{\scAccuracy} \supset \bigcap_{\sampleSize \in \Nz} \badSetU{\sampleSize}{\scAccuracy}.
	\end{align*}
	This means that
	\begin{align*}
		 \prodProb{\probMeasure}{\infty} \left( \set[\Big]{ \sample{i}{+\infty} \in \uncertainSet \suchthat \lim_{\sampleSize \to +\infty} \sopOptSolution[\sampleSize] \leq \ropOptSolution - 2\scAccuracy} \right) &\geq \prodProb{\probMeasure}{\infty}\left(\bigcap_{\sampleSize \in \Nz} \badSetU{\sampleSize}{\scAccuracy}\right) \\
		&= \inf_{\sampleSize \in \Nz} \prodProb{\probMeasure}{\sampleSize}\left(\badSetU{\sampleSize}{\scAccuracy}\right).
	\end{align*}
	However, in view of \eqref{eq: bad set def} and the fact that the $\uncertainParam_i$'s are sampled independently, 
	\begin{align*}
		\prodProb{\probMeasure}{\sampleSize}\big(\badSetU{\sampleSize}{\scAccuracy}\big) &= \prodProb{\probMeasure}{\sampleSize}\left(\bigcup_{\scOptElement \in \scOptSet} \bigcap_{i = 1}^{\sampleSize} \set[\Big]{ \sample{i}{\sampleSize} \in \uncertainSet^{\sampleSize} \suchthat \scCost(\scOptElement, \uncertainParam_i) \leq \ropOptSolution - \scAccuracy}\right), \\
		& \geq \sup_{\scOptElement \in \scOptSet} \prodProb{\probMeasure}{\sampleSize}\left(\bigcap_{i = 1}^{\sampleSize} \set[\Big]{ \sample{i}{\sampleSize} \in \uncertainSet^{\sampleSize} \suchthat \scCost(\scOptElement, \uncertainParam_i) \leq \ropOptSolution - \scAccuracy}\right), \\
		& \geq \sup_{\scOptElement \in \scOptSet} \left( \prodProb{\probMeasure}{}\left(\set[\Big]{ \sample{i}{\sampleSize} \in \uncertainSet^{\sampleSize} \suchthat \scCost(\scOptElement, \uncertainParam_i) \leq \ropOptSolution - \scAccuracy}\right) \right)^\sampleSize, \\
		& \geq \sup_{\scOptElement \in \scOptSet} \bigl( 1-\tailProb(\scOptElement, \scAccuracy) \bigr)^\sampleSize = 1 \ \text{ by our assumption.} 
	\end{align*}
	In other words,
	\begin{align*}
		\prodProb{\probMeasure}{\infty}\left(\set[\Big]{ \sample{i}{+\infty} \in \uncertainSet \suchthat \lim_{\sampleSize \to +\infty} \sopOptSolution[\sampleSize] \leq \ropOptSolution - \scAccuracy} \right) = 1,
	\end{align*}
	and the assertion follows.
\end{proof}
\begin{remark}
It is clear from the proof of Theorem \ref{th:weak nec} that the set $ \set{ \uncertainParam \in \uncertainSet \suchthat \scCost(x, \uncertainParam) > \ropOptSolution - \epsilon} $ is precisely the  set from which one needs to sample in order to get a solution with $\scAccuracy$ accuracy; if the sampled sequence $\sample{i}{+\infty}$ does not contain any element from the set $ \set{ \uncertainParam \in \uncertainSet \suchthat \scCost(x, \uncertainParam) > \ropOptSolution - \epsilon} $ corresponding to any one value of $\scOptElement$, then the approximate solution $\sopOptSolution$ is going to be atleast $\scAccuracy$ far away from $\ropOptSolution$. In the light of this fact, the condition that $\infTailProb(\scAccuracy) = 0$ amounts to saying that the sets from which one needs to sample in order to get an approximation of accuracy $\scAccuracy$ are arbitrarily small regions of $\uncertainSet$; consequently and in restrospect the above result appears to be natural.
\end{remark}

\begin{remark}
The function $\infTailProb(\cdot)$ is very similar to an object we have encountered before: cf. the function $\infTailProbWors(\cdot)$ defined in \eqref{eq:inf tail prob worse}. These two functions $\infTailProb$ and $\infTailProbWors$ are weakly related to each other. In general, one can say that $\infTailProb(\scAccuracy) \leq \infTailProbWors(\scAccuracy)$ for every $\scAccuracy > 0$. Indeed, observe that for each $\scOptElement \in \scOptSet$, since $\ropOptSolution \leq \marginalFunction(\scOptElement),$  \[ \set{ \uncertainParam \in \uncertainSet \suchthat \scCost(x, \uncertainParam) > \ropOptSolution - \epsilon} \subset \set[\Big] { \uncertainParam \in \uncertainSet \suchthat \scCost(\scOptElement, \uncertainParam) > \marginalFunction(\scOptElement) - \scAccuracy }  \] since $\ropOptSolution \leq \marginalFunction(\scOptElement)$. This implies that $\tailProb(\scOptElement, \scAccuracy) \leq \tailProbWors(\scOptElement, \scAccuracy)$ for each $\scOptElement \in \scOptSet$ and $\scAccuracy > 0$, which further implies that \[ \infTailProb(\scAccuracy) = \inf_{\scOptElement \in \scOptSet} \tailProb(\scOptElement, \scAccuracy) \leq \inf_{\scOptElement \in \scOptSet} \tailProbWors(\scOptElement, \scAccuracy) = \infTailProbWors(\scAccuracy). \] In subsequent sections (see Remark \ref{rem: inf tail prob fundamental}) we will see further evidence pointing to the fundamental nature of $\infTailProb(\cdot)$ in relation to the scenario approach .

\end{remark}

\begin{example}[continues=exa:cont]
We check whether the type of obstruction introduced in Theorem \ref{th:weak nec} arises in the situation given in  example \ref{eg:inconsistent}. Since we know that $\ropOptSolution = 0$, if we take $\scAccuracy = 1, $ we get
\begin{align*}
\set{ \uncertainParam \in \uncertainSet \suchthat \scCost(x, \uncertainParam) > \ropOptSolution - \epsilon} &= \set{ \uncertainParam \in \R \suchthat \scCost(x, \uncertainParam) > -1} \\
&= \set{ \uncertainParam \in \R \suchthat \uncertainParam \leq \scOptElement - 1}, 
\end{align*}
which implies that \footnotemark 
\begin{equation*}
\tailProb(\scOptElement, 1) = \probMeasure(\set{ \uncertainParam \in \uncertainSet \suchthat \uncertainParam \leq \scOptElement - 1}) = \erfc(1-\scOptElement), 
\end{equation*}
and therefore,
\begin{equation*}
\infTailProb(1) = \inf_{\scOptElement \in \R} \erfc(1-\scOptElement) = 0.
\end{equation*}
It is now evident that the obstruction pointed out by Theorem \ref{th:weak nec} that prevents consistency in this example as well.
\end{example}
\footnotetext{Recall that the function $ \R{} \ni z \mapsto \erfc(z) \Let \frac{2}{\sqrt{\pi}} \int_{z}^{+\infty} \exp(-t^2) \dd t$ is equal to the $\probMeasure(\set{ \uncertainParam \in \R \suchthat \uncertainParam \leq -z })$ when $\probMeasure$ is the standard Gaussian measure (normal with mean 0 and variance 1) on $\R$. Clearly, $\lim_{z \to +\infty} \erfc(z) = \inf_{z \in \R} \erfc(z) = 0.$} 

The result of Theorem \ref{th:weak nec} is equally valid in the case where $\scOptSet$ is compact. However, we started the discussion claiming that the obstruction arises when $\scOptSet$ is a noncompact set, and the next proposition affirms this statement: we show that when the set $\scOptSet$ is compact, the obstruction presented in Theorem \ref{th:weak nec} cannot arise.

\begin{proposition}
\label{p:compact consistent}
	Consider the problem \eqref{eq:ROP} along with its associated data. If Assumption \ref{ass:full support} holds, then for every $\scAccuracy > 0, \ \tailProb(\cdot, \scAccuracy) : \scOptSet \lra [0, 1]$ is a positive lower semi continuous function, and consequently, for each $\scAccuracy > 0$,
	\begin{equation*}
	 \infTailProb(\scAccuracy) = \inf_{\scOptElement \in \scOptSet} \tailProb(\scOptElement, \scAccuracy) = \min_{\scOptElement \in \scOptSet} \tailProb(\scOptElement, \scAccuracy) > 0.
	\end{equation*}
\end{proposition}
A proof of Proposition \ref{p:compact consistent} is provided in Appendix \ref{app:proofs}.
	
	\section{Finite sample performance guarantees in the nonconvex setting}
		\label{sec:positive}
In this section, for a large class of nonconvex minmax problems, we prove a general positive result that gives an upper bound on the a priori probability of the bad set \eqref{eq: bad set def} for finite samples of the scenario approach. In other words, we establish a finite sample performance guarantee in a general nonconvex setting. Of course, in the presence of more detailed structure, we may be able to refine these preliminary estimates, and as an illustration of this scheme we then discusss several special cases of this result.

\subsection{General performance guarantees}
The first order of business is making the word \emph{nonconvexity} precise. The class of nonconvex functions is vast, and it appears that very little can be said about a priori estimates under the scenario method at this level of generality; indeed, it is natural to expect, at least in principle, that the greater the regularity of the functions under consideration, tighter the bounds that should be possible to obtain. Physical considerations point us towards focussing our investigations on classes of functions that arise naturally in physical systems, e.g., trigonometric polynomials of finite bandwidth, smooth functions restricted to compact sets, etc. Our approach here follows standard principles of functional analysis and approximation theory via estimates involving covering numbers \`a la \cite{ref:CukZho-07}; the techniques exposed here are fairly general, and conform to the following simple steps:
\begin{tcolorbox}[title=Summary of our approach]
	\begin{enumerate}[label=\textbf{(\Roman*)}, leftmargin=*, widest=II, align=left]
		\item We find upper bounds on the covering number of the family of functions
			\[
				\scRvSet \Let \set[\big]{\scCost(\scOptElement, \cdot) : \uncertainSet \lra \R \suchthat \scOptElement \in \scOptSet}
			\]
			in the supremum norm topology defined below. This step provides us with a finite collection of representatives from the (possibly infinite dimensional) class of functions under consideration.
		\item The i.i.d property of the sampling in the scenario approach permits us to employ the bounds on the covering number found in the preceding step in standard probabilistic inequalities to arrive at bounds on the probability $\badSetProb{\sampleSize}{\scAccuracy}$.
	\end{enumerate}
\end{tcolorbox}

\subsubsection{Background}
The class of nonconvex functions is vast, and we consider only a few reasonable classes of finite and infinite dimensional subsets of this class in the article at hand. The primary difficulty with infinite dimensionality of function classes is overcome in a standard way by the consideration of covering numbers. Recall that given a metric space $(M, d)$ and a subset $K \subset M$, a set $K' \subset K$ is called an $\scAccuracy$ cover of $K$ if for every element $a \in K$, there exist $a' \in K'$ such that $d(a, a') \leq \scAccuracy$. We define the \emph{covering number} $\coveringNum{K}{\scAccuracy}$ of $K$ to be the smallest number $n \in \Nz$ such that there exists an $\scAccuracy$ cover of $K$ of cardinality $n$. It is a standard result that $K \subset M$ is precompact iff for all $\epsilon > 0$, the covering number $\coveringNum{K}{\epsilon}$ is finite. 

Recall that if $\uncertainSet$ is compact, the set $\smoothFnSet{}{(\uncertainSet)}$ of continuous real valued functions on $\uncertainSet$ is a metric space when endowed with the metric $d(g_1, g_2) \Let \unifnorm{g_1 - g_2}$ inherited from the supremum norm given by $ \unifnorm{g} \Let \sup_{\uncertainParam \in \uncertainSet} |g(\uncertainParam)|  $.

\subsubsection{Main Result}
The following theorem is the key result of this section. Given the problem \eqref{eq:ROP} and its associated notation, let $\scRvSet$ denote the family of functions 
\begin{equation}
	\label{eq: Kf def}
	\scRvSet \Let \set[\big]{\scCost(\scOptElement, \cdot) : \uncertainSet \lra \R \suchthat \scOptElement \in \scOptSet}.
\end{equation}

\begin{theorem}
	\label{th:gen prop}
	Consider the problem \eqref{eq:ROP} along with its associated data, and suppose that Assumption \ref{ass:full support} holds. Let \(\scRvSet\) be as defined in \eqref{eq: Kf def}, and recall from \eqref{eq: bad set prob} that $\badSetProb{\sampleSize}{\scAccuracy} = \prodProb{\probMeasure}{\sampleSize}(\badSet{\sampleSize}{\scAccuracy})$. If $\uncertainSet$ is compact and the set of functions $\scRvSet \subset \smoothFnSet{}{(\uncertainSet)}$ defined in \eqref{eq: Kf def} is precompact in $\smoothFnSet{}{(\uncertainSet)}$, then
	\begin{equation}
		\badSetProb{\sampleSize}{\scAccuracy} \leq \coveringNum{\scRvSet}{\frac{\scAccuracy}{4}} \exp\Big(-\sampleSize \infTailProb\Big(\frac{\scAccuracy}{4}\Big)\Big).
	\end{equation}
\end{theorem}
\begin{proof}
	Fix $\scAccuracy > 0$ and $m \in \N$. By definition of $\coveringNum{\scRvSet}{\frac{\scAccuracy}{4}}$, there exists a subset $(\scOptElement_i )_{i=1}^{\coveringNum{\scRvSet}{\frac{\scAccuracy}{4}}}$ of $\scOptSet$ such that for each $\scOptElement \in \scOptSet$ there exists $\scOptElement_i$ that satisfies
	\begin{equation*}
	\sup_{\uncertainParam \in \uncertainSet} | \scCost(\scOptElement, \uncertainParam) - \scCost(\scOptElement_i, \uncertainParam) | < \frac{\scAccuracy}{4}.
	\end{equation*}
	Recalling the definition of $\badSetU{\sampleSize}{\scAccuracy}$ from \eqref{eq: bad set def}, we see that
	\begin{align*}
	\badSetU{\sampleSize}{\frac{\scAccuracy}{2}} &= \bigcup_{\scOptElement \in \scOptSet} \bigcap_{i = 1}^{\sampleSize} \set[\Big]{ \sample{i}{\sampleSize} \in \uncertainSet^{\sampleSize} \suchthat \scCost(\scOptElement, \uncertainParam_i) \leq \ropOptSolution - \frac{\scAccuracy}{2}} \\ 
	& \subset \bigcup_{j = 1}^{\coveringNum{\scRvSet}{\frac{\scAccuracy}{4}}} \bigcap_{i = 1}^{\sampleSize} \set[\Big]{ \sample{i}{\sampleSize} \in \uncertainSet^{\sampleSize} \suchthat \scCost(\scOptElement_j, \uncertainParam_i) \leq \ropOptSolution - \frac{\scAccuracy}{4} }.
	\end{align*}
	This means, in view of the definition of $\badSetProb{\sampleSize}{\scAccuracy}$ in \eqref{eq: bad set prob}
	\begin{align*} 
	\badSetProb{\sampleSize}{\scAccuracy} &= \prodProb{\probMeasure}{\sampleSize}(\badSet{\sampleSize}{\scAccuracy}) \leq \prodProb{\probMeasure}{\sampleSize}\left(\badSetU{\sampleSize}{\frac{\scAccuracy}{2}}\right) \\
	&\overset{(\dag)}{\leq}  \sum_{j = 1}^{\coveringNum{\scRvSet}{\frac{\scAccuracy}{4}}} \prodProb{\probMeasure}{\sampleSize}\left(\bigcap_{i = 1}^{\sampleSize} \set[\Big]{ \sample{i}{\sampleSize} \in \uncertainSet^{\sampleSize} \suchthat \scCost(\scOptElement_j, \uncertainParam_i) \leq \ropOptSolution - \frac{\scAccuracy}{4} }\right) \\
	& \leq \sum_{j = 1}^{\coveringNum{\scRvSet}{\frac{\scAccuracy}{4}}} \prodProb{\probMeasure}{}\left(\set[\Big]{ \uncertainParam \in \uncertainSet \suchthat \scCost(\scOptElement_j, \uncertainParam) \leq \ropOptSolution - \frac{\scAccuracy}{4} }\right) ^{\sampleSize} \\
	& \leq \sum_{j = 1}^{\coveringNum{\scRvSet}{\frac{\scAccuracy}{4}}} \left(1 - \prodProb{\probMeasure}{}\left(\set[\Big]{ \uncertainParam \in \uncertainSet \suchthat \scCost(\scOptElement_j, \uncertainParam) > \ropOptSolution - \frac{\scAccuracy}{4} }\right) \right)^{\sampleSize} \\
	& \overset{(\ddag)}{\leq} \sum_{j = 1}^{\coveringNum{\scRvSet}{\frac{\scAccuracy}{4}}} \exp{ \left( -\sampleSize \tailProb\left(\scOptElement_j, \frac{\scAccuracy}{4}\right) \right)} \\
	& \leq \coveringNum{\scRvSet}{\frac{\scAccuracy}{4}} \exp{ \left( -\sampleSize \infTailProb\left(\frac{\scAccuracy}{4}\right) \right)},
	\end{align*}
	as asserted, where we have employed the standard union bound in step \((\dag)\) and the inequality \(1 - z \le \epower{-z}\) for \(z\in[0, 1]\) in step \((\ddag)\) above.
\end{proof}

\begin{remark}
	In the absence of any further structure in the various sets that appear in the definition of \(\badSet{\sampleSize}{\scAccuracy}\) in the proof of Theorem \ref{th:gen prop}, it appears that the standard union bound employed in step \((\dag)\) in the proof above is a reasonable option. However, in certain specific cases it may be possible to refine this particular step further to arrive at tighter bounds.
\end{remark}

\begin{remark}\label{rem: pos alt form}
Theorem \ref{th:gen prop} can be rewritten in the following way. Suppose we are given a desired accuracy level $\scAccuracy$ and confidence level $\scConf$. We define
\begin{align}\label{eq: sample size}
\sampleSize(\scAccuracy, \scConf) \Let \frac{1}{\infTailProb(\frac{\scAccuracy}{4})} \left( \ln\left(\frac{1}{\scConf}\right) + \ln{\coveringNum{\scRvSet}{\frac{\scAccuracy}{4}}} \right).
\end{align}
If the number $\sampleSize$ of i.i.d samples (scenarios) drawn from $\uncertainSet$ under $\probMeasure$ is at least $\sampleSize(\scAccuracy, \scConf)$, then the probability of occurence of the bad set $\badSet{\sampleSize}{\scAccuracy}$ is guaranteed to be at most $\scConf$. Let us compare this estimate with that of Theorem \ref{th:convex sop ftg}. We begin by noting that $N(\scAccuracy, \scConf)$ defined in \eqref{eq:def N} can be rewritten explicitly as (see \cite[Theorem 1]{ref:CamGarPran-09})
\begin{equation}\label{eq:def N explicit}
N(\scAccuracy, \scConf) \geq \frac{2}{\scAccuracy}\left( \ln\left(\frac{1}{\scConf}\right) + \scOptDim \right).
\end{equation}
	With this in mind, we can rewrite the result of Theorem \ref{th:convex sop ftg} in the language of its statement as follows: If the number $\sampleSize$ of i.i.d samples drawn from $\uncertainSet$ under $\probMeasure$ is at least $N(\scAccuracy, \scConf)$, then the probability of occurence of the bad set $\badSet{\sampleSize}{h(\scAccuracy)}$ is guaranteed to be at most $\scConf$. Observe that $N(\scAccuracy, \scConf)$ samples guarantee an accuracy of $h(\scAccuracy)$ as opposed to $\scAccuracy$. To remove this dependency on $h(\scAccuracy)$ and to obtain an explicit result, recall the definitions in \eqref{eq:tail prob worse}, \eqref{eq:inf tail prob worse} and \eqref{eq: ULB}. If $h(\cdot)$ has a well defined inverse $h^{-1}(\cdot)$, then sampling $N(h^{-1}(\scAccuracy), \scConf)$ number of elements would guarantee an accuracy of $\scAccuracy$ for each \(\scAccuracy\). However such an inverse function may not always exist. Neverthless, the right hand side of \eqref{eq: ULB} in the definition of $h(\cdot)$ is a pseudo inverse of the function $\infTailProbWors$. Indeed, if $\infTailProbWors$ is monotonically increasing, then $\sup \set[\big] { \delta \geq 0 \suchthat \inf_{\scOptElement \in \scOptSet} \tailProbWors(\scOptElement, \scAccuracy) \leq \delta }$ is its inverse in the usual sense. Due to this fact, if we employ $\infTailProbWors$ as the inverse of $h(\cdot)$, define
\begin{align}
	\label{eq: convex sample size}
	\convSampleSize(\scAccuracy, \scConf) \Let \frac{2}{\infTailProbWors(\scAccuracy)} \left( \ln\left(\frac{1}{\scConf}\right) + \scOptDim \right),
\end{align}
and sample $\convSampleSize(\scAccuracy, \scConf)$ number of i.i.d.\ samples from $\uncertainSet$, then Theorem \ref{th:convex sop ftg} guarantees that the probability of occurence of the bad set $\badSet{\sampleSize}{\scAccuracy}$ is at most $\scConf$. 
\end{remark}

\begin{remark}
	\label{rem: inf tail prob fundamental}
	Theorem \ref{th:weak nec} along with the estimate of Theorem \ref{th:gen prop} in the form given in \eqref{eq: sample size} points to the fundamental nature of the function $\infTailProb(\cdot)$. On the one hand, Theorem \ref{th:weak nec} says that if $\infTailProb(\scAccuracy)$ is equal to zero, then the scenario approach is not consistent and even as the number of i.i.d samples drawn approach infinity, the scenario approximation remains at least $\scAccuracy$ away from the true minimum. On the other hand, according to Theorem \ref{th:gen prop}, even if $\infTailProb(\scAccuracy)$ is nonzero, the number of samples that need to be drawn to guarantee an approximation of accuracy $\scAccuracy$ grows increasingly large as $\infTailProb(\scAccuracy)$ goes to zero. This is reminiscent of the \emph{condition number} of a matrix in linear algebra; recall that a square matrix is singular only if its condition number is infinite. However, even if the condition number is finite, it becomes increasingly harder to numerically compute the inverse of a matrix as its condition number increases, to the extent that a matrix with a very large condition number is practically singular from a numerical standpoint. In this sense, $\infTailProb(\scAccuracy)$ is a measure of how well behaved the scenario approximations of a robust optimization problem are: as $\infTailProb(\scAccuracy)$ decreases, the finite sample behaviour of the scenario approximations also deteriorate, and finally, when $\infTailProb(\scAccuracy)$ becomes zero, the performance deteriorates so much that even consistency is lost.
\end{remark}

\subsection{Scenario bounds for bandlimited trigonometric functions}
Here we employ Theorem \ref{th:gen prop} of the previous section to derive bounds on the probability of the ``bad set'' $\badSet{\sampleSize}{\scAccuracy}$ in the situation where $\uncertainSet$ is an $\uncertainDim$-dimensional hypercube and the set of functions $\scRvSet = \{\scCost(\scOptElement, \cdot) : \uncertainSet \lra \R \}_{\scOptElement \in \scOptSet} \subset \smoothFnSet{}{(\uncertainSet)}$ is a bounded subset of the linear subspace of trigonometric polynomials of bandwidth $\order$. More precisely, our premise for this subsection is the following:

\begin{assumption}
	\label{ass:trigonometric poly}
	In the context of the problem \eqref{eq:ROP} and its associated data, we stipulate that:
	\begin{enumerate}[label=\textup{(\roman*)}, align=right, widest=iii, leftmargin=*]
		\item $\uncertainSet = [-\hypercubeLen, \hypercubeLen]^{\uncertainDim} \subset \R^{\uncertainDim}$,
		\item For each $\scOptElement \in \scOptSet$, the function $\scCost(\scOptElement, \cdot) : \uncertainSet \lra \R$ is a trigonometric polynomial of bandwidth $\order$. In other words,
		\begin{equation*}
			\scCost(\scOptElement, \uncertainParam) = \sum_{k \in [-\order, \order]^{\uncertainDim} \cap \Z^{\uncertainDim}} \coeffSin_{k}(\scOptElement) \sin\left(2\pi \inprod{k}{\uncertainParam}\right) + \coeffCos_{k}(\scOptElement) \cos\left(2\pi \inprod{k}{\uncertainParam}\right),
		\end{equation*}
		\item $\sup_{\scOptElement \in \scOptSet} \ltwonorm{\scCost(\scOptElement, \cdot)} \teL \trigLtwoBound < +\infty$.
	\end{enumerate}
\end{assumption}

The set of trigonometric polynomials of bandwidth $\order$ is a $(2\order+1)^{\uncertainDim}$ dimensional subspace of $\smoothFnSet{}{(\uncertainSet)}$, and therefore, any bounded subset of it is precompact. Consequently, Theorem \ref{th:gen prop} applies to this situation. In the following Lemma, whose proof is deferred to Appendix \ref{app:cov num est}, we provide estimates of the covering number $\coveringNum{\scRvSet}{\scAccuracy}$. 

\begin{lemma}
	\label{l:trig covering num}
	Suppose that Assumption \ref{ass:trigonometric poly} holds. Let \(\scRvSet\) be as defined in \eqref{eq: Kf def} and define $\dimTrig \Let (2\order+1)^{\uncertainDim}$. Then
	\[
		\coveringNum{\scRvSet}{\scAccuracy} \leq \frac{1}{\dimTrig}\biggl(\frac{\pi \dimTrig^2}{2}\biggr)^{\dimTrig} {\left(\frac{\scAccuracy}{2\trigLtwoBound}\right)}^{-2\dimTrig}.
	\]
\end{lemma} 

Lemma \ref{l:trig covering num} in conjunction with Theorem \ref{th:gen prop} give us the following bound on the probability of the ``bad set'' $\badSet{\sampleSize}{\scAccuracy}$ defined in \eqref{eq: bad set def}.

\begin{theorem}
	\label{th:trig final}
	Consider the problem \eqref{eq:ROP} along with its associated data and suppose that Assumption \ref{ass:full support} holds. In addition, suppose that the family of functions $\scRvSet$ defined in \eqref{eq: Kf def} satisfies Assumption \ref{ass:trigonometric poly}, and define $\dimTrig \Let (2\order+1)^{\uncertainDim}$. Then, for $\badSetProb{\sampleSize}{\scAccuracy}$ as defined in \eqref{eq: bad set prob}, we have
	\[
		\badSetProb{\sampleSize}{\scAccuracy} \leq \frac{1}{\dimTrig} \biggl(\frac{\pi \dimTrig^2}{2}\biggr)^{\dimTrig} {\Big(\frac{\scAccuracy}{8\trigLtwoBound}\Big)}^{-2\dimTrig} \exp\Big(-\sampleSize \infTailProb\Big(\frac{\scAccuracy}{4}\Big)\Big).
	\]
\end{theorem}
\begin{proof}
	The assertion follows readily after substituting the estimate for the covering number $\coveringNum{\scRvSet}{\scAccuracy}$ given in Lemma \ref{l:trig covering num} in Theorem \ref{th:gen prop}.
\end{proof}

\begin{remark}
As mentioned in Remark \ref{rem: pos alt form}, we can rewrite the result of Theorem \ref{th:trig final} in terms of the number of samples required to achieve a desired level of accuracy and confidence. As before, suppose we are given an accuracy level $\scAccuracy$ and a confidence level $\scConf$. If we draw 
\begin{align}
	\label{eq: trig sample size}
	\sampleSize(\scAccuracy, \scConf) \Let \frac{1}{\infTailProb(\frac{\scAccuracy}{4})} \left( \ln\Bigl(\frac{1}{\scConf}\Bigr) + 2\dimTrig \ln\Bigl({\frac{1}{\scAccuracy}}\Bigr) + (2\dimTrig-1) \ln{\dimTrig} + \dimTrig \ln\bigl(32\trigLtwoBound^2 \pi \bigr) \right).
\end{align}
number of i.i.d samples from $\uncertainSet$ under $\probMeasure$, then the probability of occurence of the bad set $\badSet{\sampleSize}{\scAccuracy}$ is guaranteed to be less than $\scConf$.

\end{remark}

\subsection{Scenario bounds for smooth functions on the \(n\)-torus} 

We apply Theorem \ref{th:gen prop} to the problem of determining bounds on the probability of the ``bad set'' $\badSet{\sampleSize}{\scAccuracy}$ when the uncertainty set is an $\uncertainDim$-dimensional torus $\torus{\uncertainDim}$ and the cost function \(\scCost\) is smooth with respect to the uncertain parameters. Formally, we ask:
\begin{assumption}
\label{ass:smooth periodic}
	In the context of the problem \eqref{eq:ROP} and its associated data, we stipulate that:
	\begin{enumerate}[label=\textup{(\roman*)}, align=right, widest=iii, leftmargin=*]
		\item $\uncertainSet = \torus{\uncertainDim}$;
		\item there exists an integer \(\smoothness > \frac{\uncertainDim}{2}\) such that for each $\scOptElement \in \scOptSet$, the function $\scCost(\scOptElement, \cdot) : \uncertainSet \lra \R$ is a $\smoothness$-times continuously differentiable function on $\uncertainSet$;
		\item $\sup_{\scOptElement \in \scOptSet} \ltwonorm{\scCost(\scOptElement, \cdot)} \teL \trigLtwoBound < +\infty$;
		\item $\sup_{\scOptElement \in \scOptSet} \sum_{i = 1}^{\uncertainDim} \ltwonorm{\frac{\partial^{\smoothness} \scCost(\scOptElement, \cdot)}{\partial \uncertainParam_i^{\smoothness}}} \teL \ltwoderbound < +\infty$.
	\end{enumerate}
\end{assumption}

In the following Lemma, whose proof is provided in Appendix \ref{app:cov num est}, we provide estimates of the covering number $\coveringNum{\scRvSet}{\frac{\scAccuracy}{2}}$ that show that under Assumption \ref{ass:smooth periodic}, Theorem \ref{th:gen prop} applies to $\scRvSet$. 

\begin{lemma}
	\label{l:smooth periodic covering num}
	Suppose that Assumption \ref{ass:smooth periodic} holds. Let \(\scRvSet\) be as defined in \eqref{eq: Kf def} and define
	$
		\dimTrig(\scAccuracy) \Let \left(2\left\lceil \torusConsP{12} \right\rceil+1\right)^{\uncertainDim}.
	$
	Then,
	\[
		\coveringNum{\scRvSet}{\frac{\scAccuracy}{4}} \leq \frac{1}{\dimTrig(\scAccuracy)} \biggl(\frac{\pi \dimTrig(\scAccuracy)^2}{2}\biggr)^{\dimTrig(\scAccuracy)} {\left(\frac{\scAccuracy}{24\trigLtwoBound}\right)}^{-2\dimTrig(\scAccuracy)}.
	\]
\end{lemma}

Lemma \ref{l:smooth periodic covering num} in conjunction with Theorem \ref{th:gen prop} give us the following bound on the probability of the ``bad set'' $\badSet{\sampleSize}{\scAccuracy}$.

\begin{theorem}
	\label{th:smooth periodic final}
	Consider the problem \eqref{eq:ROP} along with its associated data and suppose that Assumption \ref{ass:full support} holds. In addition, suppose that the family of functions $\scRvSet$ defined in \eqref{eq: Kf def} satisfies Assumption \ref{ass:smooth periodic}. Then for $\badSetProb{\sampleSize}{\scAccuracy}$ as defined in \eqref{eq: bad set prob}, we have
	\begin{equation}
		\badSetProb{\sampleSize}{\scAccuracy} \leq \frac{1}{\dimTrig(\scAccuracy)} \biggl(\frac{\pi \dimTrig(\scAccuracy)^2}{2}\biggr)^{\dimTrig(\scAccuracy)} {\Big(\frac{\scAccuracy}{24\trigLtwoBound}\Big)}^{-2\dimTrig(\scAccuracy)} \exp\Big(-\sampleSize \infTailProb\Big(\frac{\scAccuracy}{4}\Big)\Big),
	\end{equation}
	where $\dimTrig(\scAccuracy) = \left(2\left\lceil \torusConsP{12} \right\rceil+1\right)^{\uncertainDim}$.
\end{theorem}
\begin{proof}
	The result follows by substituting the estimate for the covering number $\coveringNum{\scRvSet}{\scAccuracy}$ given in Lemma \ref{l:trig covering num} in Theorem \ref{th:gen prop}.
\end{proof}

\begin{remark}
	As mentioned in Remark \ref{rem: pos alt form}, we can rewrite the result of Theorem \ref{th:smooth periodic final} in terms of the number of samples required to achieve a desired level of accuracy and confidence. As before, suppose we are given an accuracy level $\scAccuracy$ and a confidence level $\scConf$. If we draw 
	\[
		\sampleSize(\scAccuracy, \scConf) \Let \frac{1}{\infTailProb(\frac{\scAccuracy}{4})} \left( \ln\Bigl(\frac{1}{\scConf}\Bigr) + 2\dimTrig(\scAccuracy) \ln\Bigl({\frac{1}{\scAccuracy}}\Bigr) + (2\dimTrig(\scAccuracy)-1) \ln{\dimTrig(\scAccuracy)} + \dimTrig(\scAccuracy) \ln\bigl(72\ltwobound^2 \pi \bigr) \right).
	\]
	number of i.i.d samples from $\uncertainSet$ under $\probMeasure$, then the probability of occurence of the bad set $\badSet{\sampleSize}{\scAccuracy}$ is guaranteed to be less than $\scConf$.
\end{remark}

	\appendix

	\section{Proofs of Lemma \ref{l:ess sup equivalence} and Proposition \ref{p:compact consistent}}
		\label{app:proofs}
\begin{proof}[Proof of Lemma \ref{l:ess sup equivalence}]
We first show that $\esssup_{\uncertainParam \in \uncertainSet} \scCost(\scOptElement, \uncertainParam) \leq \sup_{\uncertainParam \in \mSupp{\uncertainSet}} \scCost(\scOptElement, \uncertainParam)$. Indeed, since
\begin{align*}
\mSupp{\uncertainSet} \subset \set[\bigg]{ \uncertainParam \in \uncertainSet \ \suchthat \ \scCost(\scOptElement, \uncertainParam) \leq \ \sup_{\uncertainParam \in \mSupp{\uncertainSet}} \scCost(\scOptElement, \uncertainParam) },
\end{align*}
we have
\begin{align*}
 \probMeasure\biggl(\set[\bigg]{ \uncertainParam \in \uncertainSet \ \suchthat \ \scCost(\scOptElement, \uncertainParam) \leq \ \sup_{\uncertainParam \in \mSupp{\uncertainSet}} \scCost(\scOptElement, \uncertainParam) } \biggr) \geq \probMeasure(\mSupp{\uncertainSet}) = 1,
\end{align*}
which implies that
\begin{align*}
\esssup_{\uncertainParam \in \uncertainSet} \scCost(\scOptElement, \uncertainParam) \leq \sup_{\uncertainParam \in \mSupp{\uncertainSet}} \scCost(\scOptElement, \uncertainParam). 
\end{align*}
Now we demonstrate that $\esssup_{\uncertainParam \in \uncertainSet} \scCost(\scOptElement, \uncertainParam) + \delta \geq \sup_{\uncertainParam \in \mSupp{\uncertainSet}} \scCost(\scOptElement, \uncertainParam)$ for any $\delta > 0$. Indeed, by definition of the essential supremum,
\begin{align*}
\probMeasure\biggl( \set[\bigg]{ \uncertainParam \in \uncertainSet \ | \ \scCost(\scOptElement, \uncertainParam) \leq \ \esssup_{\uncertainParam \in \uncertainSet} \scCost(\scOptElement, \uncertainParam) + \delta} \biggr) = 1.
\end{align*}
Lower semicontinuity of $\scCost$ shows that the set \[ \set[\bigg]{ \uncertainParam \in \uncertainSet \suchthat \scCost(\scOptElement, \uncertainParam) \leq \ \esssup_{\uncertainParam \in \uncertainSet} \scCost(\scOptElement, \uncertainParam) + \delta} \] is closed. Since $\mSupp{\uncertainSet}$ is the smallest closed set of probability 1,
\begin{align*}
\mSupp{\uncertainSet} \subset \set[\bigg]{ \uncertainParam \in \uncertainSet \ \suchthat \ \scCost(\scOptElement, \uncertainParam) \leq \ \esssup_{\uncertainParam \in \uncertainSet} \scCost(\scOptElement, \uncertainParam) + \delta},
\end{align*}
which in turn means that
\begin{equation*}
\sup_{\uncertainParam \in \mSupp{\uncertainSet}} \scCost(\scOptElement, \uncertainParam) \leq \esssup_{\uncertainParam \in \uncertainSet} \scCost(\scOptElement, \uncertainParam) + \delta.
\end{equation*}
Since the preceding statement is valid for any $\delta > 0$, we have \[ \esssup_{\uncertainParam \in \uncertainSet} \scCost(\scOptElement, \uncertainParam) \geq \sup_{\uncertainParam \in \mSupp{\uncertainSet}} \scCost(\scOptElement, \uncertainParam), \]
yielding the assertion and completing the proof.
\end{proof}
	
\begin{proof}[Proof of Proposition \ref{p:compact consistent}]
	Firstly, we show that for each fixed $\scOptElement \in \scOptSet$ the set
	\begin{equation*}
		\set[\big] { \uncertainParam \in \uncertainSet \suchthat \scCost(\scOptElement, \uncertainParam) > \ropOptSolution - \scAccuracy } \subset \uncertainSet
	\end{equation*}
	is open and nonempty.
	Since $\scCost$ is l.s.c, this set is clearly open since its complement is the sublevel set of a l.s.c function with the variable $\scOptElement$ held fixed. By definition of the supremum, there exists $\uncertainParam \in \uncertainSet$ such that $\scCost(\scOptElement, \uncertainParam) > \marginalFunction(\scOptElement) - \scAccuracy$. Along with the fact that $ \marginalFunction(\scOptElement) \geq \ropOptSolution$ for all $\scOptElement$, this implies that the set is also nonempty. By Assumption \ref{ass:full support} we see that the tail probability $\tailProb(\scOptElement, \scAccuracy) = \probMeasure \left( \set{ \uncertainParam \in \uncertainSet \suchthat \scCost(x, \uncertainParam) > \ropOptSolution - \epsilon}  \right) > 0$.
	
	Second, we show that the tail probability $\tailProb(\scOptElement, \uncertainParam)$ is l.s.c in $\scOptElement$. To this end, fix $\scAccuracy \in [0, 1]$ and observe that for each fixed $\scOptElement \in \scOptSet$,
	\begin{equation*}
		\tailProb(\scOptElement, \scAccuracy) = \probMeasure \left( \scCost(\scOptElement, \cdot) > \ropOptSolution - \scAccuracy \right) = \Expec{\indic{\set{\scCost(\scOptElement, \cdot) > \ropOptSolution - \scAccuracy}}}.
	\end{equation*} 
	Fix $\scOptElement \in \scOptSet$ and pick a sequence $(\scOptElement_n)_{n = 1}^{+\infty}$ in $\scOptSet$ converging to $\scOptElement$, i.e., $ \lim_{n \to +\infty} \scOptElement_n = \scOptElement$. We claim that
	\begin{equation*}
		\liminf_{n \to +\infty} \indic{\set{\scCost(\scOptElement_n, \cdot) > \ropOptSolution - \scAccuracy}} \geq \indic{\set{\scCost(\scOptElement, \cdot) > \ropOptSolution - \scAccuracy}}.
	\end{equation*}
		Fix $\uncertainParam \in \uncertainSet$. If $ \scCost(\scOptElement, \uncertainParam) \leq \ropOptSolution - \scAccuracy$, then the preceding inequality is obvious, so let us assume that $\scCost(\scOptElement, \uncertainParam) > \ropOptSolution - \scAccuracy$. By lower semicontinuity of $f$,
		\begin{align*}
			&\liminf_{n \to +\infty} \scCost(\scOptElement_n, \uncertainParam) \geq \scCost(\scOptElement, \uncertainParam) > \ropOptSolution - \scAccuracy, 
		\end{align*}
	which means that for all $n$ sufficiently large, 
		\begin{align*}
		\scCost(\scOptElement_n, \uncertainParam) > \ropOptSolution - \scAccuracy, 
		\end{align*}
	and this in turn means
		\begin{align*}
		& \liminf_{n \to +\infty} \indic{\set{\scCost(\scOptElement_n, \cdot) > \ropOptSolution - \scAccuracy}} = 1 = \indic{\set{\scCost(\scOptElement, \cdot) > \ropOptSolution - \scAccuracy}}.
		\end{align*}
	It follows that
	\begin{equation*}
		\Expec{\liminf_{n \to +\infty} \indic{(\scCost(\scOptElement_n, \cdot) > \ropOptSolution - \scAccuracy)}} \geq \Expec{\indic{(\scCost(\scOptElement, \cdot) > \ropOptSolution - \scAccuracy)}}.
	\end{equation*}
	By Fatou's lemma \cite[Lemma 8.1, Chapter IV]{ref:Dib-16} we have
	\begin{equation*}
		\liminf_{n \to +\infty} \Expec{\indic{\set{\scCost(\scOptElement_n, \cdot) > \ropOptSolution - \scAccuracy}}} \geq \Expec{\liminf_{n \to \infty} \indic{\set{\scCost(\scOptElement_n, \uncertainParam) > \ropOptSolution - \scAccuracy}}},
	\end{equation*}
	and combining the inequalities we see that
	\begin{align*}
		&\liminf_{n \to +\infty} \Expec{\indic{\set{\scCost(\scOptElement_n, \cdot) > \ropOptSolution - \scAccuracy}}} \geq \Expec{\indic{\set{\scCost(\scOptElement, \cdot) > \ropOptSolution - \scAccuracy}}},
	\end{align*}
		thereby establishing that $\liminf_{n \to +\infty} \tailProb(\scOptElement_n, \scAccuracy) \geq \tailProb(\scOptElement, \scAccuracy).$ Since $(\scOptElement_n)_{n = 1}^{+\infty}$ and $\scOptElement$ were arbitrary, lower semicontinuity of $\tailProb(\cdot, \scAccuracy)$ follows for each fixed $\scAccuracy$. This completes the proof of the first claim.
	\par Finally, since $\tailProb(\cdot, \scAccuracy)$ is l.s.c for each fixed $\scAccuracy$, it attains its minimum on any compact subset of $\scOptSet$ by Weierstrass' theorem \cite[Theorem 7.1, Chapter II]{ref:Dib-16} , which proves the second statement of our theorem.
\end{proof}

	\section{Proofs of Lemma \ref{l:trig covering num} and Lemma \ref{l:smooth periodic covering num}}
		\label{app:cov num est}
\begin{proof}[Proof of Lemma \ref{l:trig covering num}]
We estimate the covering number $\coveringNum{\scRvSet}{\scAccuracy}$ under the conditions of Assumption \ref{ass:trigonometric poly}. To start, we define the set of trigonometric polynomial of bandwidth $\order$:
\begin{equation*}
\polynomSet \Let \set[\Bigg] { p : \uncertainSet \lra \R \suchthat p(\uncertainParam) = \sum_{k \in [-\order, \order]\cap \Z^{\uncertainDim}} a_{k} \sin\left(2\pi \inprod{k}{\uncertainParam}\right) + b_{k} \cos\left(2\pi \inprod{k}{\uncertainParam}\right) },
\end{equation*}
and define the $\ltwospace$-ball of radius $\trigLtwoBound$ in $\polynomSet$ by 
\[
	\boundPolynomSet{\trigLtwoBound} \Let \set[\big] { p \in \polynomSet \suchthat \ltwonorm{p} \leq \trigLtwoBound},
\]
where the \(\ltwospace\)-norm is defined in the standard way.

The following technical lemma is needed to prove our estimate of the covering number $\coveringNum{\scRvSet}{\scAccuracy}$:

\begin{lemma}
	\label{l: cov num subset}
	Let $(M, \genMetric)$ be a metric space and let $U \subset M$ be a subset. Then, for each $\scAccuracy > 0$,
	\[
		\coveringNum{U}{\scAccuracy} \leq 2\coveringNum{M}{\frac{\scAccuracy}{2}}.
	\]
\end{lemma}
\begin{proof}
	Pick $U \subset M$ and let $(a_i)_{i=1}^{\coveringNum{M}{\frac{\scAccuracy}{2}}} \subset M$ be an $\frac{\scAccuracy}{2}$-cover of $M$. After reordering these points if necessary, suppose that for $i = 1, \ldots, p \leq \coveringNum{M}{\frac{\scAccuracy}{2}}$, there exists a $b_i \in U$ such that $\genMetric(a_i, b_i) < \frac{\scAccuracy}{2}$ and that for each $i > p$ there exists no $b \in U$ such that $\genMetric(a_i, b) < \frac{\scAccuracy}{2}$. We claim that $(b_i)_{i=1}^{p} \subset U$ is an $\scAccuracy$-cover of $U$. Take an arbitrary element $b \in U$. By definition of a cover, there exists an $a_i$ with $i \leq p$ such that $\genMetric(a_i, b) < \frac{\scAccuracy}{2}$. We also know that $\genMetric(a_i, b_i) \leq \frac{\scAccuracy}{2}$. Then, by the triangle inequality, \[ \genMetric(b, b_i) \leq \genMetric(b, a_i) + \genMetric(a_i, b_i) \leq \scAccuracy. \] Therefore, $(b_i)_{i=1}^{p} \subset U$ is an $\scAccuracy$-cover of $U$ and consequently 
	\[
		\coveringNum{U}{\scAccuracy} \leq p \leq \coveringNum{M}{\frac{\scAccuracy}{2}}.\qedhere
	\]
\end{proof}

	Continuing with our proof of Lemma \ref{l:trig covering num}, we see that the conditions of Assumption \ref{ass:trigonometric poly} are equivalent to saying that $\scRvSet \subset \boundPolynomSet{\trigLtwoBound}$, and we know \cite[pg. 787]{ref:BasGro-04} that 
	\[
		\coveringNum{\boundPolynomSet{\trigLtwoBound}}{\frac{\scAccuracy}{2}} \leq \frac{1}{\dimTrig} \biggl(\frac{\pi \dimTrig^2}{2}\biggr)^{\dimTrig} {\left(\frac{\scAccuracy}{2\trigLtwoBound}\right)}^{-2\dimTrig}.
	\]
	These observations in conjunction with Lemma \ref{l: cov num subset} proves that
	\begin{equation*}
		\coveringNum{\scRvSet}{\scAccuracy} \leq \frac{1}{\dimTrig} \biggl(\frac{\pi \dimTrig^2}{2}\biggr)^{\dimTrig} {\left(\frac{\scAccuracy}{2\trigLtwoBound}\right)}^{-2\dimTrig},
	\end{equation*}
	as asserted.
\end{proof}

The rest of this section will be devoted to proving Lemma \ref{l:smooth periodic covering num} that provides the key estimate of the covering number $\coveringNum{\scRvSet}{\scAccuracy}$, where $\scRvSet$, the family of functions defined in \eqref{eq: Kf def}, satisfies Assumptions \ref{ass:smooth periodic}. We will need some preliminaries on Fourier analysis on tori in order to determine these estimates. 

We realize \(\torus{1}\) as the quotient \(\R/\Z\); this way we get on \(\torus{\uncertainDim}\) a measure, henceforth denoted by $\lebMeas$, induced by the Lebesgue measure on \(\R\). If $\mathcal{L}$ is the \(\R\)-vector space of measurable functions $g: \torus{\uncertainDim} \lra \R$ such that
\begin{equation*}
\int_{\torus{\uncertainDim}} |g(x)|^2 \,\dd \lebMeas  < +\infty,
\end{equation*}
then by identifying functions in \(\mathcal L\) that differ on sets of \(\lebMeas\)-measure \(0\), we get the space $\ltwospace(\torus{\uncertainDim})$ of square-integrable \(\R\)-valued functions on \(\torus{\uncertainDim}\). 
It is a standard fact that $\ltwospace(\torus{\uncertainDim})$ is a Hilbert space when equipped with the inner product
\begin{equation*}
\ltwospace(\torus{\uncertainDim}) \times \ltwospace(\torus{\uncertainDim}) \ni (g_1, g_2) \mapsto \inprod{g_1}{g_2} \Let \int_{\torus{\uncertainDim}} g_1(x)g_2(x) \,\dd \lebMeas,
\end{equation*}
and the corresponding induced $\ltwospace$-norm $\Ltwonorm{g}{\torus{\uncertainDim}}$ is defined by $\Ltwonorm{g}{\torus{\uncertainDim}} \Let \sqrt{\inprod{g}{g}}$. For a positive integer \(\smoothness\) we denote \(\smoothness\)-times continuously differentiable functions on the torus \(\torus{\uncertainDim}\) by \(\smoothFnSet{\smoothness}{(\torus{\uncertainDim})}\).

For $\multiindex = (\multiindex_1, \ldots, \multiindex_{\uncertainDim}) \in \Z^{\uncertainDim}$ the \(\multiindex\)-th Fourier coefficients of $g \in \ltwospace(\torus{\uncertainDim})$ is defined by
\begin{equation}\label{eq:fourier}
\fourierCoeff{g}{\multiindex} \Let \int_{\torus{\uncertainDim}} g(x) \,\epower{-2\pi i\inprod{\multiindex}{x}}\,\dd \lebMeas,
\end{equation}
which permit us to represent \(g\) via its Fourier series given by
\[
\fourierSer{g}(x) \Let \sum_{\multiindex \in \Z^{\uncertainDim}} \fourierCoeff{g}{\multiindex} \,\epower{2\pi i\inprod{\multiindex}{x}},
\]
where the convergence of the aforementioned sum is understood in the $\ltwospace$-norm sense. It is well known \cite{ref:DymKea-72} that when $g \in \smoothFnSet{1}{(\torus{\uncertainDim})}$, then its Fourier series converges pointwise and $g = \fourierSer{g}(x)$. Moreover, the following Plancherel identity is valid for all $g \in \ltwospace(\torus{\uncertainDim})$:
\begin{equation}\label{eq: plancherel}
\Ltwonorm{g}{\torus{\uncertainDim}}^2 = \sum_{\multiindex \in \Z^{\uncertainDim}} |\fourierCoeff{g}{\multiindex}|^2.
\end{equation}
If $g \in \ltwospace(\torus{\uncertainDim})\cap \smoothFnSet{1}{(\torus{\uncertainDim})}$, then the Fourier series of $g$ is related to that of its partial derivative $\frac{\partial{g}}{\partial \uncertainParam_j}$ along the \(j\)-th direction by the formula
\begin{equation}\label{eq:der four}
\fourierCoeff{\frac{\partial{g}}{\partial \uncertainParam_j}}{\multiindex} = -2\pi i\multiindex_j \fourierCoeff{g}{\multiindex}.
\end{equation}
If \(\smoothness\) is a positive integer and \(g\in \ltwospace(\torus{\uncertainDim})\cap \smoothFnSet{\smoothness}{(\torus{\uncertainDim})}\), then applying the preceding formula repeatedly $\smoothness$-times we get
\begin{equation}\label{eq:der four 2}
\left| \fourierCoeff{\frac{\partial^{\smoothness}{g}}{\partial \uncertainParam_j^{\smoothness}}}{\multiindex} \right| = |2\pi \multiindex_j|^{\smoothness} \left|\fourierCoeff{g}{\multiindex}\right|.
\end{equation}

We will also make use of the following result regarding covering numbers in our discussion in this section. 
\begin{lemma}\label{l:covering num 2}
Let $(M, \genMetric)$ be a metric space and let $\phi: M \lra M$ be a map satisfying
\begin{equation*}
\genMetric(a, \phi(a)) \leq \scAccuracy \text{ \ for all $a \in M$}.
\end{equation*}
If $U \subset M$ is a subset, then
\begin{equation*}
\coveringNum{U}{3\scAccuracy} \leq \coveringNum{\phi(U)}{\scAccuracy}.
\end{equation*}
\end{lemma}
\begin{proof}
Pick $U \subset M$ and let $\bigl(\phi(a_i)\bigr)_{i=1}^{\coveringNum{\phi(U)}{\scAccuracy}} \subset \phi(U)$ be an $\scAccuracy$-cover of $\phi(U)$. We demonstrate that $(a_i)_{i=1}^{\coveringNum{\phi(U)}{\scAccuracy}} \subset U$ is a $3\scAccuracy$-cover of $U$. To see this, let $a$ be an arbitrary element of $U$. Then for some $i$, we have
\begin{align*}
\genMetric(a, a_i) \leq \genMetric(a, \phi(a)) + \genMetric(\phi(a), \phi(a_i)) + \genMetric(\phi(a_i), a_i) \leq 3\scAccuracy.
\end{align*}
The assertion follows.
\end{proof}

Recall that for $\multiindex \in \Z^{\uncertainDim}$, the infinity norm is defined as $\infnorm{\multiindex} \Let \max\limits_{i=1, \ldots,\uncertainDim} |\multiindex_i|.$ For a multiindex $\multiindex$, we define \(\multiindex_{\infty}\) is any element in \(\argmax\limits_{i =1, \ldots,\uncertainDim} |\multiindex_i|\). It is a standard fact that for a given $m \in \N$, the number of $\multiindex \in \Z^{\uncertainDim}$ with $\infnorm{\multiindex} = m$ is given by 
\begin{equation}
	\label{eq: inf multi num}
	\begin{aligned}
		\constNum{m} &\Let (2m+1)^{\uncertainDim} - (2m-1)^{\uncertainDim} = \sum_{j = 0}^{\uncertainDim} {\uncertainDim \choose j}2^{\uncertainDim-j}m^{\uncertainDim-j}(1-(-1)^{j}) \\
		&= \sum_{j = 0}^{\lceil\frac{\uncertainDim}{2}\rceil} {\uncertainDim \choose {2j+1}}2^{\uncertainDim-2j}m^{\uncertainDim-2j-1} \leq 4^{\uncertainDim}m^{\uncertainDim-1}.
	\end{aligned}
\end{equation}
Recalling the definitions of $\smoothness$ and $\uncertainDim$ in Assumption \ref{ass:smooth periodic}, we will have occasion to employ the fact that
\begin{equation}
\label{eq:def c2}
\sum_{m > N} \frac{1}{m^{2\smoothness-\uncertainDim+1}} \leq \frac{2}{N^{2\smoothness-\uncertainDim}}.
\end{equation}
To see why this is true, observe that
\begin{align}
	\sum_{m > N} \frac{N^{2\smoothness-\uncertainDim}}{m^{2\smoothness-\uncertainDim+1}} & \leq \sum_{m \geq N} \frac{N^{2\smoothness-\uncertainDim}}{m^{2\smoothness-\uncertainDim+1} } = \frac{1}{N}\sum_{m \geq 0} \frac{N^{2\smoothness-\uncertainDim+1}}{\left(N+m\right)^{2\smoothness-\uncertainDim+1} }\nonumber\\
	& = \frac{1}{N}\sum_{m \geq 0} \frac{1}{\left(1+\frac{m}{N}\right)^{2\smoothness-\uncertainDim+1}}.\label{eq:def c3}
\end{align}
The right-hand side of \eqref{eq:def c3} is the upper Darboux sum of the function $x \mapsto \left(\frac{1}{1+x}\right)^{2\smoothness-\uncertainDim+1}$, and hence decreases with increasing $N$; in particular, if $N = 1$, it is equal to $\sum_{m \geq 1}\frac{1}{m^{2\smoothness-\uncertainDim+1}}$. Consequently,
\begin{equation}
\label{eq:def c4}
\sum_{m > N} \frac{N^{2\smoothness-\uncertainDim}}{m^{2\smoothness-\uncertainDim+1}} \leq \sum_{m \geq 1}\frac{1}{m^{2\smoothness-\uncertainDim+1}} \overset{(\dag)}{\leq} \sum_{m \geq 1}\frac{1}{m^2} \leq 2,
\end{equation}
where the inequality in step \((\dag)\) follows from the assumption that $2\smoothness-\uncertainDim \geq 1$. \eqref{eq:def c2} now follows directly from \eqref{eq:def c4}.

We begin our proof of Lemma \ref{l:smooth periodic covering num} with the following Lemma. 

\begin{lemma}
	\label{l: freq lim approx torus}
	Consider the problem \eqref{eq:ROP}, and suppose that the set of functions $\scRvSet$ defined in \eqref{eq: Kf def} and the uncertainty set $\uncertainSet$ satisfy Assumption \ref{ass:smooth periodic}. Then, for every $\scAccuracy > 0$, if $N \in \Z$ is picked such that $N \geq \torusCons$, then
	\begin{equation*}
	\sup_{\scOptElement \in \scOptSet} \unifnorm{\scCost(\scOptElement, \cdot) - \sum_{\multiindex \in [-N, N]^{\uncertainDim}\cap\Z^{\uncertainDim}} \fourierCoeff{\scCost(\scOptElement, \cdot)}{\multiindex}\,\epower{2\pi i\inprod{\multiindex}{\cdot}}} \leq \scAccuracy. 
	\end{equation*}
\end{lemma}

\begin{proof}
For any $\scOptElement \in \scOptSet$ and $\uncertainParam \in \uncertainSet$,
\begin{multline*}
	\left| \scCost(\scOptElement, \uncertainParam) - \sum_{\multiindex \in [-N, N]^{\uncertainDim}\cap\Z^{\uncertainDim}} \fourierCoeff{\scCost(\scOptElement, \cdot)}{\multiindex}\epower{2\pi i\inprod{\multiindex}{\uncertainParam}} \right| = \left|\sum_{\infnorm{\multiindex} > N} \fourierCoeff{\scCost(\scOptElement, \cdot)}{\multiindex}\epower{2\pi i\inprod{\multiindex}{\cdot}}  \right| \\
	\leq \sum_{\infnorm{\multiindex} > N} \left|\fourierCoeff{\scCost(\scOptElement, \cdot)}{\multiindex}\right| = \sum_{\infnorm{\multiindex} > N} \left| \fourierCoeff{\frac{\partial^{\smoothness}{\scCost(\scOptElement, \cdot)}}{\partial \uncertainParam_{\multiindex_\infty}^{\smoothness}}}{\multiindex} \right| \left| \frac{1}{\left(2\pi\infnorm{\multiindex}\right)^{\smoothness}} \right|,
\end{multline*}
	where the last equality follows from \eqref{eq:der four 2}. From the Schwarz inequality and Assumption \ref{ass:smooth periodic} we get the estimate
\begin{multline}
	\sum_{\infnorm{\multiindex} > N} \left| \fourierCoeff{\frac{\partial^{\smoothness}{\scCost(\scOptElement, \cdot)}}{\partial \uncertainParam_{\multiindex_\infty}^{\smoothness}}}{\multiindex} \right| \left| \frac{1}{\left(2\pi\infnorm{\multiindex}\right)^{\smoothness}} \right| \\
	\leq \sqrt{\sum_{\infnorm{\multiindex} > N} \left| \fourierCoeff{\frac{\partial^{\smoothness}{\scCost(\scOptElement, \cdot)}}{\partial \uncertainParam_{\multiindex_\infty}^{\smoothness}}}{\multiindex} \right|^2}\sqrt{\sum_{\infnorm{\multiindex} > N}\left| \frac{1}{\left(2\pi \infnorm{\multiindex}\right)^{\smoothness}} \right|^2} \\
	\leq \sqrt{\sum_{\infnorm{\multiindex} > N} \left( \sum_{1 \leq j \leq \uncertainDim} \left| \fourierCoeff{\frac{\partial^{\smoothness}{\scCost(\scOptElement, \cdot)}}{\partial \uncertainParam_{j}^{\smoothness}}}{\multiindex} \right|\right)^2}\sqrt{\sum_{\infnorm{\multiindex} > N}\left| \frac{1}{\left(2\pi\infnorm{\multiindex}\right)^{\smoothness}} \right|^2} \\
	\leq \sum_{1 \leq j \leq \uncertainDim} \sqrt{\sum_{\infnorm{\multiindex} > N} \left| \fourierCoeff{\frac{\partial^{\smoothness}{\scCost(\scOptElement, \cdot)}}{\partial \uncertainParam_j^{\smoothness}}}{\multiindex} \right|^2}\sqrt{\sum_{\infnorm{\multiindex} > N}\left| \frac{1}{\left(2\pi \infnorm{\multiindex}\right)^{\smoothness}} \right|^2} \\
	\leq \ltwoderbound \sqrt{\sum_{\infnorm{\multiindex} > N}\left| \frac{1}{\left(2\pi\infnorm{\multiindex}\right)^{\smoothness}} \right|^2}.
	\label{eq:flat 2}
\end{multline}
To estimate the sum on the right hand side of \eqref{eq:flat 2}, we recall the definition of $\constNum{m}$ in \eqref{eq: inf multi num} and arrive at
\begin{equation}
\label{eq:flat 3}
\begin{aligned}
	\sqrt{\sum_{\infnorm{\multiindex} > N}\left| \frac{1}{\infnorm{\multiindex}^{\smoothness}} \right|^2} & = \sqrt{\sum_{m > N} \constNum{m} \frac{1}{m^{2\smoothness}}} \leq \sqrt{4^{\uncertainDim} \sum_{m > N} \frac{1}{m^{2\smoothness-\uncertainDim+1}}}\\
	& \leq \sqrt{\frac{4^{\uncertainDim}2}{N^{2\smoothness-\uncertainDim}}}.
\end{aligned}
\end{equation}

Putting all the above estimates together, we get
\begin{equation}
\label{eq:flat 5}
\sup_{\scOptElement \in \scOptSet} \unifnorm{\scCost(\scOptElement, \cdot) - \sum_{\multiindex \in [-N, N]^{\uncertainDim}\cap\Z^{\uncertainDim}} \fourierCoeff{\scCost(\scOptElement, \cdot)}{\multiindex}\epower{2\pi i\inprod{\multiindex}{\cdot}}} \leq \frac{\sqrt{2}\ltwoderbound 2^{\uncertainDim}}{(2\pi)^{\smoothness}} \sqrt{\frac{1}{N^{2\smoothness-\uncertainDim}}}.
\end{equation}
Since $N \geq \torusCons$, by hypothesis, \eqref{eq:flat 5} gives us
\begin{equation}
\sup_{\scOptElement \in \scOptSet} \unifnorm{\scCost(\scOptElement, \cdot) - \sum_{\multiindex \in [-N, N]^{\uncertainDim}\cap\Z^{\uncertainDim}} \fourierCoeff{\scCost(\scOptElement, \cdot)}{\multiindex}\epower{2\pi i\inprod{\multiindex}{\cdot}}} \leq \scAccuracy,
\end{equation}
proving the assertion.
\end{proof}

We are finally ready for the Proof of Lemma \ref{l:smooth periodic covering num}.
\begin{proof}[Proof of Lemma \ref{l:smooth periodic covering num}]
Pick \(\scAccuracy > 0\), and consider the family of functions 
\[ 
	\scRvSet^{\scAccuracy} \Let \set[\Bigg]{\sum_{\multiindex \in [-N, N]^{\uncertainDim}\cap\Z^{\uncertainDim}} \fourierCoeff{\scCost(\scOptElement, \cdot)}{\multiindex}\epower{2\pi i\inprod{\multiindex}{\cdot}} : \uncertainSet \rightarrow \R \suchthat \scOptElement \in \scOptSet},
\]
where $N = \left\lceil \torusConsP{12}\right\rceil$ with the various constants as defined in Lemma \ref{l: freq lim approx torus} and the discussion before it. Lemma \ref{l: freq lim approx torus} in conjunction with Lemma \ref{l:covering num 2} gives us
\begin{equation}
\coveringNum{\scRvSet}{\frac{\scAccuracy}{4}} = \coveringNum{\scRvSet^{\scAccuracy}}{\frac{\scAccuracy}{12}}.
\end{equation}
Observe that by the Plancherel's identity \eqref{eq: plancherel}, for each $\scOptElement \in \scOptSet$ we have
\begin{align*}
\sum_{\multiindex \in [-N, N]^{\uncertainDim}\cap\Z^{\uncertainDim}} |\fourierCoeff{\scCost(\scOptElement, \cdot)}{\multiindex}|^2 &\leq \sum_{\multiindex \in \Z^{\uncertainDim}} |\fourierCoeff{\scCost(\scOptElement, \cdot)}{\multiindex}|^2 \leq \Ltwonorm{\scCost(\scOptElement, \cdot)}{\torus{\uncertainDim}} \leq \ltwobound.
\end{align*}
This means that $\scRvSet^{\scAccuracy}$ is a family of bandlimited trigonometric polynomials with bounded $\ltwospace$-norm. Consequently, Lemma \ref{l:trig covering num} applies to $\scRvSet^{\scAccuracy}$, and we have the estimate
\begin{equation*}
\coveringNum{\scRvSet}{\frac{\scAccuracy}{2}} \leq \coveringNum{\scRvSet^{\scAccuracy}}{\frac{\scAccuracy}{12}} \leq  \frac{1}{\dimTrig(\scAccuracy)}{\biggl(\frac{\pi \dimTrig(\scAccuracy)^2}{2}\biggr)}^{\dimTrig(\scAccuracy)} {\left(\frac{\scAccuracy}{12\trigLtwoBound}\right)}^{-2\dimTrig(\scAccuracy)},
\end{equation*}
where $\dimTrig(\scAccuracy) = \left(2\left\lceil \torusConsP{12} \right\rceil+1\right)^{\uncertainDim}$. This proves the assertion, thereby completing our proof.
\end{proof}

	\bigskip

\end{document}